\documentclass[reqno]{amsart}
\usepackage{amsmath,amsthm,amssymb}

\usepackage{pstricks, pst-plot, pst-node, pst-grad, pst-math}
\usepackage{graphicx,color}

\usepackage{chngcntr}

\counterwithin{figure}{section}

\newtheorem{theorem}{Theorem}[section]

\newtheorem{proposition}[theorem]{Proposition}
\newtheorem{corollary}[theorem]{Corollary}

\theoremstyle{definition}

\newtheorem{example}[theorem]{Example}

\theoremstyle{remark}
\newtheorem{remark}[theorem]{Remark}

\newcommand\lm{\lambda}

\newcommand\al{\alpha}
\newcommand\be{\beta}
\newcommand\de{\delta}

\newcommand{\rd}{{\rm d}}

\newcommand{\Ran}{\mathop{\rm Ran}}
\newcommand{\Ker}{\mathop{\rm Ker}}

\newcommand{\N}{{\mathbb N}}
\newcommand{\R}{{\mathbb R}}
\newcommand{\C}{{\mathbb C}}
\newcommand{\Z}{{\mathbb Z}}

\newcommand\eps{\varepsilon}
\newcommand\beq{\begin{equation}}
\newcommand\eeq{\end{equation}}
\newcommand{\beqnt}{\begin{equation*}}
\newcommand{\eeqnt}{\end{equation*}}
\newcommand{\dom}{\mathcal{D}}
\newcommand{\dist}{\mathrm{dist}}
\newcommand\re{\mathrm{Re}\hspace{0.3mm}}
\newcommand\im{\mathrm{Im}\hspace{0.3mm}}
\newcommand\I{\mathrm{i}}
\newcommand{\set}[2]{\{#1 : #2 \}}

\newcommand{\SSet}[2]{\left\{  #1 : #2 \right\}  }

\newcommand{\sigmaess}{\sigma_{\mathrm{ess}}}

\renewcommand\H{\mathcal{H}}

\newcommand\A{\mathcal{A}}

\newcommand\rref[1]{{\rm \ref{#1}}}

\numberwithin{equation}{section}

\renewcommand{\emptyset}{\varnothing}

\begin{document}

\title[Non-symmetric perturbations of self-adjoint operators]{Non-symmetric perturbations of \\ self-adjoint operators}

\author{Jean-Claude Cuenin} 
\address{Mathematisches Institut, Ludwig-Maximilians-Universit\"at M\"unchen, Theresienstr.~39, 80333~M\"unchen, Germany}
\email{cuenin@math.lmu.de}

\author{Christiane Tretter}  
\address{Mathematisches Institut, Universit\"at Bern, Sidlerstr.~5, 3012~Bern, Switzerland}
\email{tretter@math.unibe.ch}

\date{\today}
\subjclass{47A10; 47A55, 81Q12}
\keywords{Spectrum, perturbation theory, non-selfadjoint operator, spectral gap, resolvent estimate, essential spectrum, Dirac operator, periodic system, Hamiltonian.}

\begin{abstract}
We investigate the effect of non-symmetric relatively bounded per\-turbations on the spectrum of self-adjoint operators. In particular, we estab\-lish stability theorems for one or infinitely many spectral gaps along~with corresponding resolvent estimates. These results extend, and improve, classical perturbation results by Kato and by Gohberg/Kre\u\i n. Further, we~study essential spectral gaps and perturbations exhibiting additional structure with respect to the unperturbed operator; in the latter case, we can even allow~for perturbations with relative bound $\ge\!1$. The generality of our results is illustra\-ted by several applications, massive and massless Dirac operators,  point-cou\-pled periodic systems, and two-channel Hamiltonians \vspace{-3.9mm}with~dissipation.
\end{abstract}

\maketitle

\section{Introduction}

Analytical information about the spectra and resolvents of non-self-adjoint linear operators is of great importance for numerical analysis and 
corresponding non-linear problems. Recent papers on spectral problems for non-selfadjoint differential operators have emphasized the need for universal information on the non-real spectrum and on eigenvalues in spectral gaps as well as for resolvent estimates in spectral gaps, see e.g.\ \cite{MR2727810}, \cite{MR3084346}, \cite{MR2198968}. Moreover, operators with spectral gaps have been featuring in modern applications such as periodic quantum graphs or photonic crystals, see e.g.\ \cite{MR3323566}, \cite{MR3291353}, \cite{ELT15}. However, even for perturbations of self-adjoint operators there are only a few general results, usually restricted to bounded or symmetric or relatively compact perturbations, or more specific results e.g.\ for perturbations of Schr\"odinger operators. 

In this paper we require neither of these conditions and study the behaviour of the spectrum under perturbations that are merely relatively bounded. 
Our main results concern the stability of spectral gaps and corresponding resolvent estimates, estimates of the non-real spectrum, the behaviour of essential spectral gaps and of isolated parts of the discrete spectrum, and the effects of additional structures of the perturbation. All results are formulated in terms of the relative boundedness constants of the perturbation; e.g.\ for the case of infinitely many spectral gaps, we establish conditions on the lengths of the spectral gaps and bands ensuring that infinitely many spectral gaps remain open or, more strongly, at most finitely many spectral gaps close. Our results on structured perturbations seem to be the first that even allow for perturbations with relative bound $\ge\!1$.
  
The outline of the paper is as follows. In Section \ref{sec:2} we study the effect of a relatively bounded perturbation $A$ on the spectrum $\sigma(T)$ of a self-adjoint 
operator~$T$. We show that, if $A$ has $T$-bound $\de_A<1$, then the non-real part of $\sigma(T+A)$ lies between two hyperbolas and we establish a ``gap condition" ensuring that a spectral gap $(\al_T,\beta_T)\subset \R$ of $T$ gives rise to a stable spectral free strip of $T+A$. This means that there exists a non-empty  subinterval $(\al_{T+A},\be_{T+A})\subset (\al_T,\beta_T)$ such \vspace{-0.5mm} that 
\begin{equation}
\label{jan03}
 \sigma(T+sA) \cap \{ z\in\C: \al_{T+A} < \re z < \be_{T+A} \} = \emptyset \quad \mbox{for all } s\in[0,1];
\vspace{-0.5mm}
\end{equation}  
an analogous result is proved for essential spectral gaps.
Moreover, we derive a resolvent estimate for $T+A$ in this spectral free strip. The shape of the hyperbolas, the gap condition, the bounds $\al_{T+A}$, $\be_{T+A}$, and the resolvent estimate are all formulated in terms of the relative boundedness constants of $A$ with respect to $T$ and the endpoints $\al_T$, $\be_T$ of the unperturbed spectral gap. 

Similar spectral estimates for form-bounded perturbations were proved in \cite{Veselic}; however, for non-symmetric perturbations there is no general relation between relative boundedness and relative form-boundedness. We also mention that our results extend, and improve, classical perturbation results by Kato and by Goh\-berg/Krein, see \cite[Theorems V.4.10/11]{Ka}, \cite[Lemma V.10.1]{GK69}.
 
In Section~\ref{Section Infinitely many spectral gaps} we study the stability of infinitely many spectral gaps $(\al_n,\be_n)$ of $T$ which tend to $\infty$. 
We derive conditions on $\al_n$, $\beta_n$ ensuring that $T+A$ has infinitely many stable spectral free strips 
or that, more strongly, at most finitely many spectral gaps of $T$ close under the perturbation $A$. 
A necessary condition for the latter is that the spectral gap 
lengths $l_n\!=\!\be_n\!-\!\al_n$ diverge if $A$ has $T$-bound $\de_A\!=\!0$ and that they diverge exponentially if $\de_A\!>\!0$. 
These results also apply if two spectral gaps are separated by a single spectral point, e.g.\ if $T$ has compact~resolvent.
 
In Section \ref{Section Structured perturbations} we focus on perturbations that exhibit different additional structures with respect to the unperturbed operator $T$, e.g.\ if $T$ commutes with a self-adjoint involution $\tau$, then $A$ is supposed to anti-commute with $\tau$, or vice versa. Using operator matrix techniques, we are able to tighten the spectral estimates derived in Section \ref{sec:2} and, at the same time, weaken the gap condition to such an extent that we can even allow for perturbations $A$ with $T$-bound $\de_A \ge 1$. 

For the special case of symmetric perturbations, we complement \eqref{jan03} by showing that if e.g.\ $\al_{T+A}\!=\!\al_T\!+\!\de_{T+A}$, then $\sigma(T\!+\!A)\cap\big(\al_T\!-\!\de_{T+A},\al_T\!+\!\de_{T+A}\big)\ne \emptyset$ and that if $T$ has eigenvalues of total  multiplicity $m<\infty$ in an essential spectral gap $(\al_T,\beta_T)$, then $T+A$ has eigenvalues of total algebraic multiplicity  $\le m$ in $(\al_{T+A},\be_{T+A})$. Further, we prove monotonicity results for spectral gaps and essential spectral gaps for a semi-bounded perturbation $A$.

Finally, in Section \ref{Section Applications} we apply our results to Dirac operators, massless in $\R^2$ and massive with Coulomb-like potentials in $\R^3$, to point-coupled periodic systems on manifolds, and to two-channel scattering systems with dissipation.

The following notation will be used throughout the paper. For a closed linear operator $T$ on a Hilbert space $\H$ with domain $\dom(T)$, we denote 
the kernel and range by $\Ker T$ and $\Ran T$, respectively, and the spectrum~and re\-solvent set by $\sigma(T)$ and $\rho(T)$, respectively.
Moreover, $T$ is called Fredholm if $\Ker T$ is finite dimensional and $\Ran T$ is finite co-dimensional; the essential spectrum of $T$ is defined as $\sigmaess(T):=\{ \lm\in \C: T-\lm \mbox{ is not Fredholm}\}$. 
If $T\!=\!T^*$ is self-adjoint and $J\!\subset\!\R$ is an interval, $E_T(J)$ denotes the corresponding spectral projection. 


\section{Perturbation of spectra and spectral gaps}
\label{sec:2}

In this section we study non-symmetric relatively bounded perturbations of self-adjoint operators and their effect on the spectrum. In particular, 
we estimate the non-real spectrum and the change of spectral gaps under such perturbations. All spectral enclosures are supplied with corresponding resolvent estimates. 

If $T$ and $A$ are linear operators in a Banach or Hilbert space, then 
$A$ is called \emph{$T$-bounded} if  $\dom(T) \subset \dom(A)$ and there exist $a'$,~$b'\geq 0$ such that
\begin{equation}
\label{c1} 
  \quad \|Ax\| \le a'\|x\| + b'\|Tx\|, \quad x\in \dom(T).
\end{equation}
The infimum $\delta_A$ of all $b'\ge 0$ such that there is an $a'\ge 0$ with \eqref{c1} 
or, equivalently, the infimum $\delta_A$ of all $b\ge 0$ such that there is an $a\ge 0$ with
\beq\label{c2} 
  \|Ax\|^2\leq a^2\|x\|^2+b^2\|Tx\|^2, \quad x\in\dom(T),
\eeq
is called \emph{$T$-bound of~$A$} (see \cite[Section IV.1.1]{Ka}). Note that \eqref{c2} implies \eqref{c1} with $a'=a$, $b'=b$, while \eqref{c1} implies \eqref{c2} with $a^2=a'^2(1+1/\eps)$, $b^2=b'^2(1+\eps)$ for arbitrary $\eps>0$.

Classical perturbation theorems of Kato for spectra of self-adjoint operators $T$ either assume that the perturbation $A$ is bounded or that $T$ is semi-bounded and $A$ is symmetric (see e.g.\ \cite[Theorems V.4.10/11]{Ka}). A much less known theorem of Gohberg and Krein assumes that $A$ is relatively compact (see \cite[Lemma V.10.1]{GK69}).

The following new result requires neither of these conditions.

\begin{theorem}
\label{new-non-symmetric gap}
Let $T$ be a self-adjoint operator in a Hilbert space $\H$ and let $A$ be $T$-bounded with $T$-bound $<1$ and with $a$,~$b \ge 0$, $b<1$, 
\vspace{1mm} as in \eqref{c2}.

\begin{enumerate}
\item[{\rm i)}] Then the  spectrum of $\,T+A$ lies between two hyperbolas, more precisely,
\beq\label{eq. general sa}
\sigma(T+A) \cap \bigg\{ z\in {\mathbb C} : |\im z |^2 > \displaystyle{\frac{a^{2} + b^{2} |\re  z|^2}{1-b^{2}}} \bigg\} = \emptyset.
\eeq

\item[{\rm ii)}]  If $\,T$ has a spectral gap $(\alpha_T,\beta_T)\subset \R$, i.e.\ $\sigma(T) \cap (\alpha_T,\beta_T) = \emptyset$ 
with $\alpha_T$, $\beta_T \in \sigma(T)$, and if
\beq\label{non symmetric gap remains}
 \sqrt{a^2+b^2\alpha_T^2} + \sqrt{a^2+b^2\beta_T^2} < \beta_T - \alpha_T,
\eeq
then $\,T+A$ has a stable spectral free strip $(\alpha_{T+A},\beta_{T+A}) + \I\R \subset \C$, i.e. 
\beq
\label{new non symmetric gap}
\sigma(T+ s A)\cap
\SSet{z\in\C}{\alpha_{T+ A} < \re z < \beta_{T+A}}
= \emptyset, \quad s\in[0,1],
\eeq
with 
\begin{equation}
\label{spectral-shifts}
  \alpha_{T+A} := \alpha_T+\sqrt{a^2+b^2\alpha_T^2}, \quad \beta_{T+A}:=\beta_T-\sqrt{a^2+b^2\beta_T^2};
\end{equation}
if $A$ is symmetric, then
\begin{align}
\label{eq. lower semicont gap 2} 
  \hspace*{14mm} & \sigma(T\!+\!A) \cap K_{\sqrt{a^2\!+\!b^2\al_T^2}}(\al_T) 
  \!\ne\! \emptyset, \quad
  \sigma(T\!+\!A) \cap  K_{\sqrt{a^2\!+\!b^2\be_T^2}}(\be_T) 
  \!\ne\! \emptyset. \vspace{-3mm}
\end{align}

\item[{\rm iii)}] If $\,T$ has an essential spectral gap $(\alpha_T,\beta_T)\subset \R$, i.e.\ $\sigma_{\rm ess}(T) \cap (\alpha_T,\beta_T) = \emptyset$ 
with $\alpha_T$, $\beta_T \in \sigma_{\rm ess}(T)$, and \eqref{non symmetric gap remains} holds, then the strip
\beq
\label{new essential gap}
\sigma(T+A)\cap
\SSet{z\in\C}{\alpha_{T+ A} < \re z < \beta_{T+A}}
\eeq
consists of at most countably many isolated eigenvalues of finite algebraic multiplicity
which may accumulate at most at the points $\alpha_{T+A}$, $\beta_{T+A}$.
\end{enumerate}
\end{theorem}

Resolvent estimates accompanying the spectral enclosures in Theorem \ref{new-non-symmetric gap} i) and ii) may be found in Proposition~\ref{new-resolvent estimates} below.

\begin{proof}[{\bf Proof of Theorem \ref{new-non-symmetric gap}}]
i) Since $T$ is self-adjoint, we have the resolvent estimates (see \cite[V.(3.16), (3.17)]{Ka}) 
\beq
\label{resolvent estimates s.a.}
\begin{aligned}
\|(T-z)^{-1}\|&=\frac{1}{\dist(z,\sigma(T))}\leq \frac 1{|\im z|},
\\ 
\|T(T-z)^{-1}\|&=\sup_{t\in\sigma(T)}\frac{|t|}{|t-z|} \leq\frac{|z|}{|\im z|},
\end{aligned}
\quad z\in\C\setminus \R.
\eeq
Therefore, for $z\in\C$ belonging to the second set in \eqref{eq. general sa},
\begin{equation}
\label{ref1}
\|A(T-z)^{-1}\|^2 \leq
a^2\|(T-z)^{-1}\|^2+b^2\|T(T-z)^{-1}\|^2\leq \frac{a^2+b^2|z|^2}{|\im z|^2} <1.
\end{equation}
Now \eqref{eq. general sa} follows from the stability result for bounded invertibility (see \cite[Theorem IV.3.17]{Ka}).

ii)
Let $z = \mu + \I \nu$ with $\mu \in (\alpha_T,\beta_T) \subset \rho(T)$ and $\nu\in\R$. Then $z \in \rho(T)$, 
\beq\label{T(T-nu)inv}
 \|(T-\mu)(T-(\mu+\I\nu))^{-1}\|\leq\sup_{t\in\sigma(T)}\frac{|t-\mu|}{\sqrt{|t-\mu|^2+\nu^2}}\leq 1,
\eeq
and hence
\begin{align}
\label{ineq1}
  \|A(T\!-\!z)^{-1}\| 
  \leq \|A(T\!-\!\mu)^{-1}\| \,\|(T\!-\!\mu)(T\!-\!(\mu+\I\nu))^{-1}\|
  \leq \|A(T\!-\!\mu)^{-1}\|. 
\end{align}
For $x\in \dom (T)$, we have $\|Tx\|=\| \,|T| x\|$ and hence, by \eqref{c2}, 
\[
\|Ax\|^2 \le a^2 \|x\|^2 + b^2 \|\, |T| x\|^2 = \big( (a^2 + b^2 |T|^2) x,x \big) = \big\| \sqrt{a^2 + b^2 |T|^2} \,x \big\|^2.
\]
This yields that, for all $s\in [0,1]$, 
\begin{equation}
\label{ref2a}
\begin{array}{rl}
 \| sA(T\!-\!\mu)^{-1} \| \le \|A(T\!-\!\mu)^{-1} \| 
 \!\!\!&\le \big\|\sqrt{a^2 \!+\! b^2 |T|^2}\, (T-\mu)^{-1} \big\|  \\
 &= \sup\limits_{t\in\sigma(T)} \!\dfrac{\sqrt{a^2 \!+\! b^2t^2}}{|t-\mu|}.
\end{array} 
\end{equation}
If $\re z \!=\! \mu \!\in\! \big( \alpha_T \!+\! \sqrt{a^2\!+\!b^2 \alpha_T^2}, \beta_T \!-\! \sqrt{a^2\!+\!b^2 \beta_T^2} \big)$,
then an elementary comput\-ation shows that 
\begin{equation}
\label{ref2b}
 \sup_{t\in\sigma(T)}\frac{\sqrt{a^2+b^2t^2}}{|t-\mu|}=\max\left\{b,\frac{\sqrt{a^2\!+\!b^2\al_T^2}}{\mu-\al_T},\frac{\sqrt{a^2\!+\!b^2\be_T^2}}{\be_T-\mu}\right\}<1.
\end{equation}
Now \eqref{new non symmetric gap} follows from \cite[Theorem IV.3.17]{Ka}.

For the case that $A$ is symmetric, we first prove that if $\lm\in\sigma(T)$, then
\begin{equation}
\label{eq:HS}
 \sigma(T+A)\cap(\lm-\delta_{A,\lm},\lm+\delta_{A,\lm})\neq\emptyset, 
 \quad 
 \delta_{A,\lm}:=\mathop{\limsup_{\nu\in\R\setminus\{0\}}}_{\nu\to 0}|\nu|\,\|A(T-\lm-\I\nu)^{-1}\|.
\vspace{-1mm} 
\end{equation}
If $\lm \in \sigma(T+A)$, there is nothing to prove, so we suppose that $\lm \not\in \sigma(T+A)$. 
Since $T$ is self-adjoint, there exists a singular sequence for $T$ and~$\lm$, i.e.\ a sequence $(u_n)_{n\in\N}\subset \dom(T)$, $\|u_n\|= 1$, with $\|(T-\lm)u_n\|\to 0$, $n\to\infty$ (see e.g.\ \cite[Satz~8.24~b)]{Weid}). For every $\nu\in\R\setminus\{0\}$ we have $\lm+\I\nu\in\rho(T)$~and
\[
T+A-\lm=\big(I+\left(\I\nu+A\right)\left(T-\lm-\I\nu\right)^{-1}\big)\left(T-\lm-\I\nu\right),
\]
which implies that 
\begin{equation}
\label{nikolaus1}
\|(T+A-\lm)u_n\|\leq \left(1+\|(\I\nu+A)(T-\lm-\I\nu)^{-1}\|\right) (\|(T-\lm)u_n\|+|\nu|).
\end{equation}
Similarly as above, we obtain that
\begin{equation}
\label{nikolaus2}
 \|A (T-\lm-\I\nu)^{-1}\| \le \| \sqrt{a^2\!+\!b^2|T|^2} (T-\lm-\I\nu)^{-1} \| = 
 \!\!\sup_{t\in\sigma(T)} \!\sqrt{\frac{a^2\!+\!b^2t^2}{(t\!-\!\lm)^2\!+\!\nu^2}} \!=:\!  s(\nu) \!<\! \infty,
\end{equation}
since the function over which the supremum is taken is continuous on all of $\R$ as $\nu\ne 0$ and tends to $b$ for $t\to\pm\infty$. Now we choose $n_\eps\in\N$ such that
\[
 \|(T-\lm)u_{n_\eps} \| < \frac{|\nu|}{2+s(\nu)}.
\]
Using this in \eqref{nikolaus1} and $\|u_{n_\eps}\|=1$, we conclude that
\begin{align*}
  \|(T+A-\lm)u_{n_\eps}\| 
  & \leq (2+s(\nu)) \|(T-\lm)u_{n_\eps}\| + \big( 2+ \|A (T-\lm-\I\nu)^{-1}\| \big)\,|\nu|\\ 
  & \leq \big( 3|\nu| + |\nu|\,\|A (T-\lm-\I\nu)^{-1}\| \big) \|u_{n_\eps}\|
\end{align*}
and hence
\[
 \|(T+A-\lm)^{-1}\|  \ge \frac 1{3|\nu| + |\nu|\,\|A (T-\lm-\I\nu)^{-1}\|}.
\]
Taking the limes inferior over all $\nu\in\R\setminus\{0\}$ with $\nu\to 0$ on both sides, we find
\[
 \|(T+A-\lm)^{-1}\|  \ge \mathop{\liminf_{\nu\in\R\setminus\{0\}}}_{\nu\to 0} \frac 1{3|\nu| + |\nu|\,\|A (T-\lm-\I\nu)^{-1}\|} 
 = \frac 1{\delta_{A,\lm}}.
\]
Since $A$ has $T$-bound $<1$, $T+A$ is self-adjoint by the Kato-Rellich theorem (see e.g.\ \cite[Theorem~V.4.3]{Ka}) and thus
\[
  \frac{1}{\dist(\lm,\sigma(T+A))}=\|(T+A-\lm)^{-1}\|\geq \frac 1{\delta_{A,\lm}}. 
\]
This proves \eqref{eq:HS}.

Now we are ready to prove claim \eqref{eq. lower semicont gap 2}. 
One can show that for the supremum $s(\nu)$ in \eqref{nikolaus2} the point $t(\nu)$ where it is attained tends to $\lm$ for $\nu\to 0$ and thus
\begin{align*}
\delta_{A,\lm}
 \!=\!   \mathop{\limsup_{\nu\in\R\setminus\{0\}}}_{\nu\to 0} |\nu|\,\|A(T\!-\!\lm\!-\!\I\nu)^{-1}\|
&  \le \mathop{\limsup_{\nu\in\R\setminus\{0\}}}_{\nu\to 0} |\nu| \sqrt{\!\frac{a^2\!+\!b^2t(\nu)^2}{(t(\nu)\!-\!\lm)^2\!+\!\nu^2}}
   = \sqrt{a^2\!+\!b^2 \lm^2}.
\end{align*}
This, together with \eqref{eq:HS} applied to the points $\alpha_T$, $\beta_T\in\sigma(T)$, yields \eqref{eq. lower semicont gap 2}.

iii) 
Let $(\eps_n)_{n\in\N}$, $(\delta_n)_{n\in\N} \subset [0,\infty)$, be sequences with $\eps_n\to 0$, $\delta_n\to 0$, $n\to\infty$, $\al_T\!+\!\eps_n$, $\beta_T\!-\!\delta_n \!\in\! \sigma(T)$ 
and so that the spectral projection $P_n\!:=\!E_T((\al_T\!+\!\eps_n,\be_T\!-\!\delta_n))$ has finite rank; note that we can choose e.g.\ $\eps_n=0$, $n\in\N$, if $\al_T$ is no accumulation point of eigenvalues of $T$. 
If we set $\H_1\!:=\!\Ran P_n$, $\H_2\!:=\!\H_1^{\perp}=\Ker P_n$, and denote by $T_i$ are the restrictions of $T$ to $\H_i$, 
then $\dim\H_1<\infty$ and 
\[
 T=\operatorname{diag}(T_1,T_2), \quad  
 \dom(T)=\left(\dom(T)\cap\H_1\right)\oplus\left(\dom(T)\cap\H_2\right)=\dom(T_1)\oplus\dom(T_2),
\]
in $\H=\H_1\oplus\H_2$. Since $\dom(T)\subset\dom(A)$, we can also decompose 
\[
T+A=\begin{pmatrix}T_1+A_{11}&A_{12}\\A_{21}&T_2+A_{22}\end{pmatrix},\quad \dom(T+A)=\dom(T_1)\oplus\dom(T_2),
\]
in $\H=\H_1\oplus\H_2$ where $A_{11}=P_nA|_{\dom(T)\cap\H_1}$, etc. If we write 
\beqnt\label{perturbation problem essential spectral gap}
T+A=\mathcal{T}+\mathcal{A},\quad \mathcal{T}=\begin{pmatrix}T_1&0\\0&T_2+A_{22}\end{pmatrix},\quad\mathcal{A}=\begin{pmatrix}A_{11}&A_{12}\\A_{21}&0\end{pmatrix},
\eeqnt
then $\A$ is degenerate since $\dim\H_1\!<\!\infty$. It is not difficult to check that \eqref{c2}~implies 
\beq\label{AijTjbounded}
\|A_{ij}x\|^2\leq a^2\|x\|^2+ b^2 \|T_j x\|^2, \quad x\in \dom(T_j), \quad i,j=1,2,
\eeq
and, for $|\eta|$ sufficiently large, 
\[
 \|A_{12}(T_2+A_{22}-\I\eta)^{-1}\|\leq\frac{\|A_{12}(T_2-\I\eta)^{-1}\|}{1-\|A_{22}(T_2-\I\eta)^{-1}\|}
 <\infty.
\]
Thus $\A$ is $\mathcal{T}$-bounded and $\A\mathcal{T}^{-1}$ is degenerate, hence compact. Since the essential spectrum is stable under relatively compact perturbations (see e.g.\ \cite[Theorem IX.2.1]{EE}), we \vspace{-1mm} have
\[
 \sigmaess(T+A) =\, \sigmaess(\mathcal{T}+\A)=\sigmaess(\mathcal{T})=\sigmaess(T_2+A_{22})\subset\sigma(T_2+A_{22}).
\]
By construction, 
$ \sigmaess(T_1)\!=\!\emptyset$ and $\sigma(T_2)\cap(\al_T\!+\!\eps_n,\be_T\!-\!\delta_n )\!=\!\emptyset$ with $\al_T\!+\!\eps_n$, $\beta_T\!-\!\delta_n \in \sigma(T_2)$. 
Applying ii) to $T_2$ and $A_{22}$ and letting $n\to\infty$, we conclude that $\sigmaess(T+A) \cap \big( (\alpha_{T+A},\beta_{T+A}) + \I\R \big) = \emptyset$.
By i) the strip 
$(\alpha_{T\!+\!A},\beta_{T+A}) + \I\R \!\subset\! \C$ contains points of $\rho(T)$ and hence iii) follows e.g.\ from \cite[Theorem XVII.2.1]{GGK1}. 
\end{proof}

\begin{remark}
\label{Kato-const}
There is an analogue of Theorem \ref{new-non-symmetric gap} in terms of the constants $a'\!$, $b'\!$ in \eqref{c1} 
with $a'\!+b'|\alpha_T|$, $a'\!+b'|\beta_T|$ in place of $\sqrt{a^2\!+\!b^2\alpha_T^2}$, $\sqrt{a^2\!+\!b^2\beta_T^2}$ everywhere; in particular, 
\eqref{non symmetric gap remains} becomes
\beq
\label{old non symmetric gap remains}
 (a'+b'|\alpha_T|) + (a'+b'|\beta_T|) < \beta_T - \alpha_T.
\eeq
In fact, since \eqref{c1} implies \eqref{c2} with $a^2\!=\!a'^2(1+\eps)$, $b^2\!=\!b'^2(1+\frac 1 \eps)$ for arbitrary $\eps\!>\!0$, 
we can use Theorem \ref{new-non-symmetric gap} with each such pair of constants and observe that~e.g.
\[
 \beta_T-\min_{\eps>0}\sqrt{a'^2(1+\eps)+\!b'^2\Big(1+\frac 1 \eps\Big)\beta_T^2} = \beta_T - (a'\!+b'|\beta_T|),
\]
where the minimum is attained at $\eps = \frac{b'|\beta_T|}{a'}$. Note that the corresponding condition \eqref{old non symmetric gap remains} 
with the constants $a$, $b$ from \eqref{c2}, which may also be used as $a'$, $b'$ in \eqref{c1}, is only 
sufficient but not necessary for \eqref{non symmetric gap remains}. 
\end{remark}

\begin{remark}
Theorem \ref{new-non-symmetric gap} iii) remains valid if we replace $\sigmaess(T)$ by any of the sets $\sigma_{\rm e,i}(T)$, $i=1,2,3,4$, 
defined in \cite[Section IX.1]{EE} which are all stable under relatively compact perturbations. 
\end{remark}

The spectral inclusion established in Theorem \ref{new-non-symmetric gap} is illustrated in Figure~\ref{picture perturbation of spectral gap} for the typical case $a\neq 0$, $b\neq 0$. 
Note that the asymptotes of the hyperbolas in Theorem \ref{new-non-symmetric gap} i) are given by
$|\im z| = \pm \arcsin b \, |\re z|$; if $A$ is bounded, i.e.\ $a=\|A\|$, $b=0$, the hyperbolas degenerate into the lines $|\im z| = \|A\|$, 
in agreement with the classical perturbation result (comp.\ \cite[Section V.4.3]{Ka}).

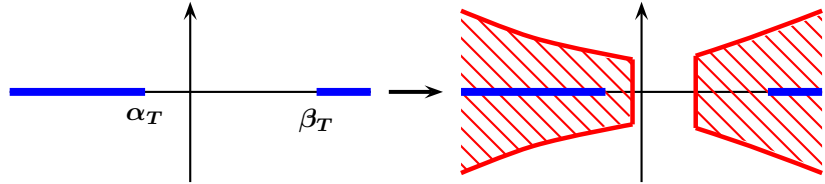
\begin{figure}[h]
 \begin{center}
 \psset{unit=.24cm} 
\begin{pspicture}(-25,0)(20,10)
\psline[linestyle=solid,arrowsize=5pt](-25,5)(-5,5)
\psline[linestyle=solid,arrowsize=5pt]{->}(-15,0)(-15,10)
\psline[linestyle=solid,linewidth=3pt,linecolor=blue](-25,5)(-17.5,5)
\psline[linestyle=solid,linewidth=3pt,linecolor=blue](-8,5)(-5,5)
\uput[d](-8,4.8){\boldmath $\beta_T$}
\uput[d](-17.5,4.8){\boldmath $\alpha_T$}
\psline[linestyle=solid,linewidth=1.5pt,arrowsize=5pt]{->}(-4,5)(-1,5) 
\psline[linestyle=solid,arrowsize=5pt](0,5)(20,5)
\psline[linestyle=solid,arrowsize=5pt]{->}(10,0)(10,10)
\pscurve[linestyle=none,linewidth=.1pt,hatchcolor=red,fillstyle=vlines]
(20,9.3)(16,7.8)(13,6.8)(13,6.3)(13,5)(13,4)(13,3.2)(16,2.2)(20,.7)
\pscurve[linestyle=none,linewidth=.1pt,hatchcolor=red,fillstyle=vlines](0,.5)(4,2)(7,2.7)(8,3.1)(9,3.3)(9.5,4)(9.5,5)(9.5,6)(9.5,6.8)(8,6.9)(7,7.3)(4,8)(0,9.5)
\psline[linestyle=solid,linewidth=1.7pt,linecolor=red](13,3)(13,7)
\psline[linestyle=solid,linewidth=1.7pt,linecolor=red](9.5,3.2)(9.5,6.8)
\pscurve[linestyle=solid,linewidth=1.7pt,linecolor=red](13,3)(16,2)(20,.5)
\pscurve[linestyle=solid,linewidth=1.7pt,linecolor=red](9.5,3.2)(4,2)(0,0.5)
\pscurve[linestyle=solid,linewidth=1.7pt,linecolor=red](13,7)(16,8)(20,9.5)
\pscurve[linestyle=solid,linewidth=1.7pt,linecolor=red](9.5,6.8)(4,8)(0,9.5)
\psline[linestyle=solid,linewidth=3pt,linecolor=blue](17,5)(20,5)
\psline[linestyle=solid,linewidth=3pt,linecolor=blue](0,5)(8,5)
\end{pspicture}
\vspace{-2mm}
\end{center}
\caption{\!\!Spectral inclusion in Theorem \ref{new-non-symmetric gap} ($a\!\neq\!0$,~$b\!\neq\!0$)\label{picture perturbation of spectral gap}.}
\label{fig:1}
\end{figure}

Theorem \ref{new-non-symmetric gap} extends, and improves, two classical perturbation results, firstly, the well-known stability theorem of T.\ Kato for symmetric perturbations of semi-bounded selfadjoint operators (see \cite[Theorem V.4.11]{Ka}) and, secondly, the less known perturbation result of M.G.\ Kre\u\i n and I.C.\ Gohberg on relatively compact perturbations of self-adjoint operators (see \cite[Lemma V.10.1]{GK69}).

The first of the following two corollaries shows that Theorem \ref{new-non-symmetric gap} ii) generalizes Kato's stability result \cite[Theorem V.4.11]{Ka} to non-symmetric $A$ and improves Kato's lower bound for the spectrum of $T+A$ even for symmetric $A$.

\begin{corollary}
\label{cor. semi-bounded}
If $\,T$ in Theorem {\rm \ref{new-non-symmetric gap} ii)} is bounded below with lower bound $\beta_T$, 
then $T+A$ is $m$-accretive with 
\begin{equation}
\label{kato-new}
 \re\hspace{0.5mm}\sigma(T+A) \ge\beta_{T+A} = \begin{cases} \beta_T - \sqrt{a^2\!+\!b^2\beta_T^2}, & \mbox{ for  $a$, $b$ as in \eqref{c2}},\\
                             \beta_T - (a'\!+b'\beta_T), & \mbox{ for  $a'\!$, $b'\!$ as in \eqref{c1}}.
               \end{cases}              
\end{equation}
\end{corollary}

\begin{proof}
The claim follows from Theorem \ref{new-non-symmetric gap} and its proof, applied with $\alpha_T$ tending to $-\infty$, and Remark \ref{Kato-const}.
\end{proof}

\begin{remark}
Corollary \ref{cor. semi-bounded} generalizes Kato's semiboundedness stability result \cite[Theorem V.4.11]{Ka} to non-symmetric perturbations. Moreover, if $A$ is symmetric it improved Kato's lower \vspace{-1mm} bound 
\[
 \sigma(T+A) \ge \beta_T - \max \Big\{ \frac{a'}{1-b'}, a' + b' |\beta_T| \Big\},
\vspace{-1mm} 
\]
which is worse than the bound in \eqref{kato-new} if $\frac{a'}{1-b'} > a' \!+\! b' |\beta_T|$, i.e.\ $a'\!+\!b'|\beta_T|> |\beta_T|$. 
\end{remark}

The next corollary shows that Theorem \ref{new-non-symmetric gap} i) does not only yield Gohberg and Kre\u\i n's result \cite[Lemma V.10.1]{GK69} as a special case where the perturbation $A$ is $T$-compact, but it generalizes their result to $T$-bounded $A$ with $T$-bound $0$.
Note that if $A$ is $T$-compact, then $A$ has $T$-bound $0$  since $H$ is reflexive and $T$ is self-adjoint, thus closed (see \cite[Corollary III.7.7]{EE}).  

\begin{corollary}
\label{cor:gk}
If $A$ in Theorem {\rm \ref{new-non-symmetric gap} i)} has 
$T$-bound $0$, then for every $\eps > 0$ there exists an $r_\eps >0$ such that
\[
  \sigma(T+A) \subset K(0,r_\eps) \cup  \Sigma_\eps \cup ( -\Sigma_\eps)
\]
where $K(0,r_\eps) := \{ z \in \C: |z|\le r_\eps\}$ is the closed ball of radius $r_\eps$ centred at $0$ and
\[
  \Sigma_\eps := \{ z\in \C \setminus\{0\} : | \arg z | \le \eps \}
\]
is the sector of opening angle $2\eps$ lying symmetrically around the positive real axis.  
\end{corollary}

\begin{proof}
Let $\eps >0$ be arbitrary. Since $A$ has $T$-bound $0$, there exist $a_\eps$, $b_\eps \ge 0$ with $b_\eps$ so small \vspace{-1mm} that
\[
 \|Ax\|^2\le a_\eps^2 \|x\|^2 + b_\eps^2 \|Tx\|^2, \ \ x\in \dom(T), \quad \mbox{ and } \quad \frac{b_\eps^2}{1-b_\eps^2} < \frac {\eps^{2}}2.
\vspace{-1mm}
\]
Now let $z\in\sigma(T+A)$. 
If $|\re z|^2 \geq r_0^2 := \frac {a_\eps^2}{1-b_\eps^2} \frac 2{\eps^{2}}$, then Theorem \ref{new-non-symmetric gap} i) shows that
\[
\frac{|\im z|^2}{|\re z|^2}\leq\frac{a_\eps^2}{1-b_\eps^2} \frac 1{r_0^2}+\frac{b_\eps^2}{1-b_\eps^2}<\eps^{2}
\]
and hence $z\in\Sigma_\eps \cup ( -\Sigma_\eps)$; if $|\re z|^2 \leq r_0^2$, then Theorem \ref{new-non-symmetric gap} i) shows that 
\[
 |z|^2 \leq r_0^2 + |\im z|^2 \leq r_0^2 + \frac{a_\eps^2+b_\eps^2 r_{0}^2}{1-b_\eps^2} =:r_\eps^{2} 
\] 
and hence $z\in K(0,r_\eps)$. Since $r_\eps \geq r_0$, the claim follows.
\end{proof}

\begin{remark}
If $T$ is self-adjoint and $A$ is \emph{$p$-subordinate} to $T$ with $0\le p<1$, i.e.\ $\dom(T) \subset \dom(A)$ and there exists $c\ge 0$ such that
\[
   \|Ax\| \le c \|x\|^{1-p} \|Tx\|^p, \quad x \in \dom(T),
\]
then $A$ is $T$-bounded with $T$-bound $0$. Hence Corollary \ref{cor:gk} implies the spectral enclosure in \cite[Lemma 3.5]{MR3169033}, which was proved there for the more general case that $T$ is bisectorial with angle $\theta \in [0,\pi/2)$ and radius $r\ge 0$, see \cite[Definition~2.7]{MR3169033}.
\end{remark}

\begin{proposition}
\label{new-resolvent estimates}
Let $T$ be a self-adjoint operator in a Hilbert space $\H$ and let $A$ be $T$-bounded with $T$-bound $<1$ and with $a$,~$b \ge 0$, $b<1$, 
\vspace{1mm} as in \eqref{c2}.

{\rm i)}  For $z\in\C$ such that $|\im z|^2 > \frac{a^{2} + b^{2} |\re  z|^2}{1-b^{2}}$, we have
\[
  \|(T+A-z)^{-1}\|\leq \frac{1}{|\im z|-\sqrt{a^2+b^2|z|^2}}.
\]
{\rm ii)}  For $z\in\C$ such that $\alpha_{T+ A} \!<\! \re z \!<\! \beta_{T+A}$ with $\alpha_{T+ A}$, $\beta_{T+A}$ as in \eqref{spectral-shifts}, we have
\[
\|(T\!+\!A\!-\!z)^{-1} \| \!\leq\!\! 
\frac1{\min\{\re z\!-\!\alpha_T,\beta_T\!-\!\re z\}^2 \!+\! (\im z)^2} \frac 1{\!\!1\!-\!\max\!\left\{\!b,\frac{\!\!\sqrt{a^2\!+\!b^2\al_T^2}\!\!}{\re z-\al_T},\frac{\!\!\sqrt{a^2\!+\!b^2\be_T^2}\!}{\be_T-\re z}\right\}\!\!}.
\]
\end{proposition}

\begin{proof}
Both claims follow from the estimates of $\|A(T-z)^{-1}\|$ established in the proofs of Theorem \ref{new-non-symmetric gap} i) and ii) and  \cite[Theorem IV.3.17]{Ka}.
\end{proof}

\vspace{-2mm}

\begin{remark}
\label{res-est-detailed}
The explicit form of the resolvent estimate in Proposition \ref{new-resolvent estimates}~ii) is different for the cases $|\alpha_T| \le |\beta_T|$ and $|\alpha_T| \ge |\beta_T|$. More precisely, if we define $\zeta \in (\alpha_{T+A},\beta_{T+A})$ by
\begin{align*}
  \zeta :\!&= \alpha_T \!+\! \frac{\beta_T-\alpha_T}{\sqrt{a^2\!+\!b^2\alpha_T^2}\!+\!\sqrt{a^2\!+\!b^2\beta_T^2}} \sqrt{a^2\!+\!b^2\alpha_T^2} \\
         &= \beta_T \!-\! \frac{\beta_T-\alpha_T}{\sqrt{a^2\!+\!b^2\alpha_T^2}\!+\!\sqrt{a^2\!+\!b^2\beta_T^2}} \sqrt{a^2\!+\!b^2\beta_T^2},
\end{align*}
then, if $|\alpha_T|\le|\beta_T|$, we have $\alpha_{T+A} \le \zeta \le \frac{\alpha_T+\beta_T}2$ and 
\begin{align*}
   & \|(T\!+\!A\!-\!z)^{-1} \| \!\leq\! \\
   &\leq \begin{cases}
   \!\dfrac1{\!\sqrt{(\re z\!-\!\alpha_T)^2 \!+\! (\im z)^2}} \dfrac {\re z - \alpha_T}{\re z \!-\! \alpha_{T\!+\!A}},\hspace{-3mm}  
     & \hspace{2mm} \re z \!\in\!\! \big( \alpha_{T+A}, \zeta \big],\\[3mm]
   \!\dfrac1{\!\sqrt{(\re z\!-\!\alpha_T)^2 \!+\! (\im z)^2}} \dfrac {\beta_T-\re z}{\beta_{T\!+\!A} \!-\! \re z},\hspace{-3mm}  
     & \hspace{2mm} \re z \!\in\!\! \big( \zeta, \min \{\frac{\alpha_T\!+\!\beta_T}2,\beta_{T+A} \} \big], \\[3mm]
   \!\dfrac1{\!\sqrt{(\beta_T\!-\!\re z)^2 \!+\! (\im z)^2}} \dfrac {\beta_T-\re z}{\beta_{T\!+\!A} \!-\! \re z},\hspace{-3mm} 
     & \hspace{2mm} \re z \!\in\!\! \big( \!\min \{ \frac{\alpha_T\!+\!\beta_T}2,\beta_{T\!+\!A} \}, \beta_{T\!+\!A} \big);
   \end{cases}
\end{align*}
if $|\alpha_T|\ge|\beta_T|$, we have $\frac{\alpha_T+\beta_T}2 \le \zeta \le \beta_{T+A}$ and 
\begin{align*}
   &\|(T\!+\!A\!-\!z)^{-1} \| \!\leq\! \\ 
   &\leq \begin{cases}
   \!\dfrac1{\!\sqrt{(\re z\!-\!\alpha_T)^2 \!+\! (\im z)^2}} \dfrac {\re z - \alpha_T}{\re z \!-\! \alpha_{T\!+\!A}},\hspace{-3mm}  
     & \hspace{2mm} \re z \!\in\!\! \big( \alpha_{T+A}, \min \{ \alpha_{T\!+\!A},\!\frac{\alpha_T\!+\!\beta_T}2 \} \big],\\[3mm]
   \!\dfrac1{\!\sqrt{(\beta_T\!-\!\re z)^2 \!+\! (\im z)^2}}  \dfrac {\re z - \alpha_T}{\re z \!-\! \alpha_{T\!+\!A}},\hspace{-3mm}  
     & \hspace{2mm} \re z \!\in\!\! \big( \min \{ \alpha_{T\!+\!A},\frac{\alpha_T\!+\!\beta_T}2 \},\zeta \big], \\[3mm]
   \!\dfrac1{\!\sqrt{(\beta_T\!-\!\re z)^2 \!+\! (\im z)^2}} \dfrac {\beta_T-\re z}{\beta_{T\!+\!A} \!-\! \re z},\hspace{-3mm} 
     & \hspace{2mm}  \re z \!\in\!\! \big( \zeta, \beta_{T\!+\!A} \big).
   \end{cases}
\end{align*}
\end{remark}

\smallskip

\begin{remark}
For perturbations in quadratic form sense, results similar to those in Theorem \ref{new-non-symmetric gap} were proved by K.\ Veselic in \cite{Veselic}, 
under corresponding assumptions on the relative form bounds of the perturbation. If $A$ is symmetric, our assumption \eqref{c2} implies that 
$|A|\leq a+b|T|$ in  quadratic form sense, which is the assumption in \cite{Veselic}. However, for non-symmetric perturbations there is no general relation between relative boundedness and relative form-boundedness.
\end{remark}

If $T$ has a spectral gap that is symmetric to the origin, the claims in Theorem~\ref{new-non-symmetric gap}~ii) 
and, in particular, the resolvent estimates in Remark \ref{res-est-detailed} simplify as follows.

\begin{corollary}
\label{symmetric gap}
Let $T$ be a self-adjoint operator in a Hilbert space $\H$ and let $A$ be $T$-bounded with $T$-bound $<1$ and with $a$,~$b \ge 0$, $b<1$ as in \eqref{c2}.
If $\,T$ has a symmetric spectral gap $(-\beta_T,\beta_T) \subset \R$ with $\beta_T>0$, i.e.\ $\sigma(T) \cap (-\beta_T,\beta_T) = \emptyset$ and $\beta_T\in\sigma(T)$, and if
\beq\label{symmetric gap remains}
 \sqrt{a^2+b^2\beta_T^2}<\beta_T,
\eeq
then $\,T\!+\!A$ has a stable spectral free strip $(-\beta_{T\!+\!A},\beta_{T\!+\!A})+\I\R \subset \C$ with $\beta_{T\!+\!A}=\beta_T\!-\! \sqrt{a^2\!+\!b^2\beta_T^2}$ as in \eqref{spectral-shifts}, more precisely,
\beq
\label{basic symmetric gap}
\sigma(T+sA)\cap\SSet{z\in\C}
{|\re z|<\beta_{T+A}}
= \emptyset, \quad s\in[0,1],
\eeq
and, for $z \in\C$ belonging to the second set in \eqref{basic symmetric gap},
\beq\label{resolvent estimate symmetric gap}
\|(T+A-z)^{-1} \|\leq \frac{1}{\sqrt{(\beta_T-|\re z|)^2+|\im z|^2}}\cdot\frac{\beta_T-|\re z|}{\beta_{T+A}-|\re z|}.
\eeq
In particular, $T+A$ is bisectorial.
\end{corollary}

Theorem \ref{new-non-symmetric gap} allows us to strengthen another classical result of Kato who showed that 
if $T$ is self-adjoint, $A$ is $T$-bounded with $T$-bound~$<\frac 12$ and constants $a'$, $b'$ as in \eqref{c1}, $\lm$ is an isolated eigenvalue of $T$ with multiplicity $m<\infty$, and
\beq\label{eigenvalue cond.}
a'+b'(|\lm|+d)<\frac{d}{2},\quad d:=\dist(\lm,\sigma(T)\setminus\{\lm\}),
\eeq
then the open ball $B(\lm,d/2)$ contains exactly $m$ eigenvalues of $T+A$, counted with algebraic multiplicity, and no other points of $\sigma(T+A)$ 
(see \cite[Section~V.4.3]{Ka}). 

\medskip

Here we can allow for perturbations with relative bound $<1$ (see Remark \ref{LAST!!!} below) and for non-symmetric spectral gaps around the isolated eigenvalue. 

\begin{theorem}
\label{eigenvalues}
Let $T$ be a self-adjoint operator in a Hilbert space $\H$, let $A$ be T-bounded with $T$-bound $<1$ and with $a,b\geq 0$, $b<1$, as in \eqref{c2}. Suppose that $\lm\in\sigma(T)$ is an isolated eigenvalue of $\,T$ of multiplicity $m<\infty$, and set
\[
 \al:=\max\SSet{\nu\in\sigma(T)}{\nu<\lm},\quad 
 \be:=\min\SSet{\nu\in\sigma(T)}{\nu>\lm}.
\]
If
\begin{equation}
\label{last!!!}
  \sqrt{a^2\!+\!b^2\al^2} \!+\! \sqrt{a^2\!+\!b^2\lm^2} < \lm \!-\! \alpha \quad \mbox{and} \quad 
  \sqrt{a^2\!+\!b^2\lm^2} \!+\! \sqrt{a^2\!+\!b^2\be^2} < \beta \!-\! \lm, 
\end{equation}  
then the vertical strip
\[
 \big( \lm- \sqrt{a^2\!+\!b^2\lm^2}, \lm + \sqrt{a^2\!+\!b^2\lm^2} \big) +\I\R
\]
contains exactly $m$ isolated eigenvalues of $\,T\!+\!A$, counted with algebraic multiplicity. 
\end{theorem}

\begin{proof}
By Theorem \rref{new-non-symmetric gap} i) there exists $\eta_0>0$ such that 
\begin{equation}
\label{luzern}
\SSet{z\in\C}{|\im z|>\eta_0}\subset\rho(T+A).
\end{equation}
Moreover, Theorem \rref{new-non-symmetric gap} ii) implies that 
\begin{equation}
\label{sets2}
\begin{split}
 &\SSet{z\in\C}{\al+\sqrt{a^2\!+\!b^2\al^2}<|\re z|<\lm-\sqrt{a^2\!+\!b^2\lm^2}}\subset\rho(T+A),\\
 &\SSet{z\in\C}{\lm+\sqrt{a^2\!+\!b^2\lm^2}<|\re z|<\be-\sqrt{a^2\!+\!b^2\be^2}}\subset\rho(T+A).
\end{split}
\end{equation}
Hence we can choose a closed rectangular Jordan curve $\Gamma \!=\!\bigcup_{i=1}^4 \Gamma_i \!\subset\! \rho(T\!+\!A)$ 
whose vertical parts $\Gamma_1$, $\Gamma_3$ pass through the two strips in \eqref{sets2} and whose horizontal parts $\Gamma_2$, $\Gamma_4$ 
lie in the two sets $\set{z\in\C}{\im z> \eta_0}$ and $\set{z\in\C}{\im z< -\eta_0}$, respectively. %

Using the estimates \eqref{ref2a}, \eqref{ref2b} in the proof of Theorem \rref{new-non-symmetric gap} in $\Gamma_1$, $\Gamma_3$ and 
the estimate \eqref{ref1} in the proof of Theorem \rref{new-non-symmetric gap} in $\Gamma_2$, $\Gamma_4$, 
one can show that the family of Riesz projections
\[
 P(\chi):=-\frac{1}{2\pi\I}\int_{\Gamma}(T+\chi A-\lm)^{-1} \rd \lm
\]
depends continuously on $\chi\in[0,1]$. Hence we have $\dim \Ran P(1)=\dim \Ran P(0)=\dim E(\{\lm\})=m$. 
This together with \eqref{luzern} yields the claim.
 \end{proof}

\begin{remark}
\label{LAST!!!}
i) Theorem \ref{eigenvalues} allows for $b\!<\!1$, and not only for $b\!<\!\frac 12$ as in \cite[Section~V.4.3]{Ka}; this can be seen e.g.\ letting $\lm\!=\!0$, $a\!=\!0$ in \eqref{last!!!} and \eqref{eigenvalue cond.},~respectively.

ii) Theorem \rref{eigenvalues} can easily be generalized to a finite number of eigenvalues separated from the rest of the spectrum. 
\end{remark}

\section{Infinitely many spectral gaps}
\label{Section Infinitely many spectral gaps}

In this section we apply the stability result for a single spectral gap to study the behaviour of infinitely many spectral gaps.
We establish criteria on the gap lengths ensuring that, under the perturbation, infinitely many spectral gaps are retained or, more strongly, at most a finite number of (finite) spectral gaps closes.

\begin{theorem}
\label{cor. infinitely many gaps nonzero rel. bound}
Let $T$ be a self-adjoint operator in a Hilbert space $\H$ with infinitely many spectral gaps, i.e.\ $\sigma(T) \cap (\alpha_{n},\beta_{n}) \!=\! \emptyset$ with $\alpha_n$, $\beta_n \!\in\! \sigma(T)$, such that $\al_n\!\!<\!\be_n\!\!\leq\!\al_{n+1}$, $n\!\in\!\N$, 
$\alpha_n\!\!\to\!\infty$, $n\!\to\!\infty$, and let $A$ be $T\!$-bounded with $T\!$-bound $\delta_A\!<\!1$. 

\begin{enumerate}
\item[{\rm i)}] If the unperturbed spectral gaps $(\alpha_n,\beta_n)$ satisfy
\begin{equation}
\label{abconst.}
 \limsup_{n\to\infty} \frac{\be_n}{\alpha_n} > \frac {1+\delta_A}{1-\delta_A} \ (\ge \!1),
\end{equation}
then $T+A$ has infinitely many stable spectral free strips; if even
\begin{equation}\label{abconst.-new}
 \liminf_{n\to\infty} \hspace{0.5mm} \frac{\be_n}{\alpha_n} > \frac {1+\delta_A}{1-\delta_A} \ (\ge\!1),
\end{equation}%
then at most finitely many spectral gaps of $\,T\!$ close under perturbation \vspace{1mm} by~$A$. 
\item[\rm ii)] If $a_n$,\,$b_n \!\ge\! 0$, $b_n\!<\!1$, are so that \eqref{c2} holds with $a_n$,\,$b_n$ in place of \vspace{-1mm} $a$,\,$b$~and 
\begin{equation}
\label{c4-newer}
  \liminf_{n\to\infty} \frac{\sqrt{a_n^2 + b_n^2 \alpha_n^2}+\sqrt{a_n^2 + b_n^2 \beta_n^2}}{\beta_n-\alpha_n} < 1,
\end{equation}
then $T+A$ still has infinitely many stable spectral free strips; if even
\begin{equation}
\label{c4-new}
  \limsup_{n\to\infty} \frac{\sqrt{a_n^2 + b_n^2 \alpha_n^2}+\sqrt{a_n^2 + b_n^2 \beta_n^2}}{\beta_n-\alpha_n} < 1,
\end{equation}
then at most finitely many spectral gaps of $\,T\!$ close under perturbation by~$A$. 
\end{enumerate}
\end{theorem}

\begin{proof}
Without loss of generality, we may assume that $\alpha_n>0$, $n\in \N$. \vspace{1mm}

i) Suppose that \eqref{abconst.} holds. Then there exists s subsequence $(n_k)_{k\in\N} \subset \N$ such that 
$\frac {\beta_{n_k}}{\alpha_{n_k}} \to  \limsup_{n\to\infty} \frac {\beta_n}{\alpha_n} =: \gamma_{\rm s} \in (\frac {1+\delta_A}{1-\delta_A},\infty]$. 
By assumption \eqref{abconst.} and the definition of the $T$-bound $\delta_A$, we can choose $a\geq 0$ and $b\in\big(\delta_A,\frac{\gamma_{\rm s}-1}{\gamma_{\rm s}+1}\big) \,\subset (\delta_A,1)$ such that \eqref{c2} holds.
\vspace{-3mm}Then
\begin{align*}
  \lim_{k\to\infty} \frac{\sqrt{a^2 \!+\! b^2 \alpha_{n_k}^2}\!+\!\sqrt{a^2 \!+\! b^2 \beta_{n_k}^2}}{\beta_{n_k}-\alpha_{n_k}}
  &=\lim_{k\to\infty} \frac{\sqrt{\frac{a^2}{\alpha_{n_k}^2} \!+\! b^2}\!+\!\sqrt{\frac{a^2}{\alpha_{n_k}^2} \!+\! b^2 \frac{\beta_{n_k}^2}{\alpha_{n_k}^2}}}{\frac{\beta_{n_k}}{\alpha_{n_k}}-1} 
 = b \, \frac{\gamma_{\rm s}\!+\!1}{\gamma_{\rm s}\!-\!1} 
 <1.
\end{align*}
By Theorem \ref{new-non-symmetric gap} ii), there exists  $k_0 \!\in\! \N$ so that $T\!+\!A$  has infinitely many stable spect\-ral free strips $(\alpha_{T\!+\!A,n_k},\beta_{T\!+\!A,n_k}) + \I \R$ with $(\alpha_{T\!+\!A,n_k},\beta_{T\!+\!A,n_k}) \!\subset\! (\alpha_{n_k},\beta_{n_k})$,~$k \!\ge\!k_0$.

Now suppose that \eqref{abconst.-new} holds and choose $\eps>0$  such that 
\[
 \gamma_{\rm i}:=\liminf_{n\to\infty} \frac {\beta_n}{\alpha_n} > \frac {1+\delta_A(1+\eps)}{1-\delta_A(1+\eps)} > \frac {1+\delta_A}{1-\delta_A}
 \ \mbox{ and } \  
 \delta_{A,\eps} := \delta_A \frac{1+\eps}{1+\frac \eps 2} < 1.  
\]
Then there exists $n_0 \!\in\! \N$ such that $\frac {\beta_{n}}{\alpha_{n}} \!>\! \frac {1\!+\delta_A(1\!+\eps)}{1\!-\delta_A(1\!+\eps)}$, $ n\!\ge\! n_0$. 
Let $a$, $b\!\ge\! 0$, $b \!\in\! (\delta_A, \delta_{A,\eps})$, be such that \eqref{c2} holds.
Since $\alpha_n\to\infty$, there exists $N_0 \in \N$, $N_0 \ge n_0$, such that $\frac{a^2}{\alpha_n^2} < b^2 (\frac{\eps^2}4 + \eps)$.
Observing that $\frac{\beta_n}{\alpha_n} > 1$, we obtain that, for\vspace{-3mm} $n\ge N_0$,
\begin{align*}
  \frac{\sqrt{a^2 \!+\! b^2 \alpha_{n}^2}\!+\!\sqrt{a^2 \!+\! b^2 \beta_{n}^2}}{\beta_{n}-\alpha_{n}}
  & = \frac{\sqrt{\frac{a^2}{\alpha_{n}^2} \!+\! b^2}\!+\!\sqrt{\frac{a}{\alpha_{n}^2} \!+\! b^2 \frac{\beta_{n}^2}{\alpha_{n}^2}}}{\frac{\beta_{n}}{\alpha_{n}}-1}
  < b\Big(1\!+\!\frac \eps2 \Big) \frac{\frac{\beta_n}{\alpha_n}+1}{\frac{\beta_n}{\alpha_n}-1} \\
  &< b\Big(1\!+\!\frac\eps 2\Big) \frac{\frac {1+\delta_A(1+\eps)}{1-\delta_A(1+\eps)}+1}{\frac {1+\delta_A(1+\eps)}{1-\delta_A(1+\eps)}-1}
  = b\Big(1\!+\!\frac\eps 2 \Big) \frac 1{\delta_A(1+\eps)} \\
  & = b \frac 1{\delta_{A,\eps}}< 1.
\end{align*}%
By Theorem \ref{new-non-symmetric gap} ii),  $T+A$ has stable spectral free strips $(\alpha_{T+A,n},\beta_{T+A,n})+\I\R$ with $(\alpha_{T+A,n},\beta_{T+A,n}) \subset (\alpha_{n},\beta_{n})$, $n \ge N_0$.

ii) The claims in ii) are immediate from Theorem \ref{new-non-symmetric gap} since we may allow the relative boundedness constants 
$a_n$, $b_n$ to be chosen differently for each spectral gap $(\alpha_n,\beta_n)$.
\end{proof}

Note that in Theorem \ref{cor. infinitely many gaps nonzero rel. bound} two spectral gaps may be separated by a single spectral point if $\alpha_n\!=\!\beta_{n+1}$. 
In particular, Theorem \ref{cor. infinitely many gaps nonzero rel. bound} applies if the operator $T$ has compact~resolvent. 

While in Theorem~\ref{cor. infinitely many gaps nonzero rel. bound} the assumptions in i) are easier to check, the conditions 
in ii) obtained using different relative boundedness constants in each spectral gap are weaker. 
In the next proposition we analyze the implications of the latter on the
growth of the lengths of the spectral gaps; we restrict ourselves to \eqref{c4-new}.

\begin{proposition}
\label{exponential growth}
Under the assumptions of Theorem \rref{cor. infinitely many gaps nonzero rel. bound}, 
let $l_n:=\be_n-\al_n > 0$, $n\in \N$, be the length of the $n$-th spectral gap 
of~$\,T$, and let $\delta_A \in [0,1)$ be the $T$-bound of $A$. Then 
\begin{quote}
\eqref{c4-new} \ requires that \
$\begin{cases} 
(l_n)_{n\in\N} \mbox{ diverges exponentially if } \,\delta_A>0, \\ 
(l_n)_{n\in\N} \mbox{ diverges if } \,\delta_A =0 \mbox{ and $A$ is unbounded}.
\end{cases}$
\end{quote}
\end{proposition}

\begin{proof}
Without loss of \vspace{1mm} generality, we may assume that $\alpha_n\!>\!0$, $n\!\in\! \N$. 

If $\delta_A>0$, then $b_n \ge \delta_A$, $n\in\N$. Together with $\beta_n \le \alpha_{n+1}$, $n\in\N$, 
we find that \eqref{c4-new} implies
\begin{align*}
 1 > \limsup_{n\to\infty} \frac{\sqrt{a_n^2 \!+\! b_n^2 \alpha_{n}^2}\!+\!\sqrt{a_n^2 \!+\! b_n^2 \beta_{n}^2}}{\beta_{n}-\alpha_{n}} 
   & \ge \limsup_{n\to\infty} b_n \,\frac{\be_n \!+\! \al_n}{\beta_{n}\!-\!\alpha_{n}} \\
   &\ge \delta_A \limsup_{n\to\infty} \Big( 1 + 2 \frac{\alpha_n}{\alpha_{n+1}-\alpha_n}\Big).
\end{align*}
It follows that
\[
 \liminf_{n\to\infty}\frac{\al_{n+1}}{\al_n} \geq  2 \frac{\delta_A}{1-\delta_A}+1 = \frac{1+\delta_A}{1-\delta_A}>1
\]
and thus $(\al_n)_{n\in\N}$ diverges exponentially. Because $l_n=\be_n-\al_n\geq (\gamma_{\rm i}-1)\,\al_n$ and $\gamma_{\rm i}>1$ for almost all $n\in\N$, so does~$(l_n)_{n\in\N}$.

\indent
If $\delta_A=0$, we first show that $a_n\to\infty$, $n\to\infty$. Otherwise,  
\[
\|Ax\|^2\leq \limsup_{n\to\infty} a_n^2 \, \|x\|^2,\quad x\in\dom(T);
\]
since $\dom(T)\subset\H$ is dense, this would imply that $A$ is bounded, a contradiction to the assumption. 
From \eqref{c4-new} it now follows that $(l_n)_{n\in\N}$ diverges since
\[
  1> \limsup_{n\to\infty} \frac{\sqrt{a_n^2 \!+\! b_n^2 \alpha_{n}^2}\!+\!\sqrt{a_n^2 \!+\! b_n^2 \beta_{n}^2}}{\beta_{n}-\alpha_{n}} \ge
 \limsup_{n\to\infty} \frac{2 a_n}{l_n}.
 \qedhere
\]
\end{proof}

The last inequality suggests that in Proposition \ref{exponential growth} ii), it may be sufficient that the divergence of $(l_n)_{n\in\N}$ is of the same order as the divergence of $(a_n)_{n\in\N}$ when $b_n\to 0$, $n\to\infty$, e.g.\ power-like rather than exponential, possibly modulated by logarithms \cite{AEL94,ADE98, BEG03}. This is confirmed by Example \ref{cor. 2 to asymptotics} below which is used in Section \ref{subsec:7.2} for a physical application. 

In fact, in applications often the growth rate of the spectral gaps and spectral bands is known rather than that of their end-points. 
Here the following alternative formulas are useful to check the conditions of Theorem \ref{cor. infinitely many gaps nonzero rel. bound}; again we restrict ourselves to \eqref{abconst.-new} and \eqref{c4-new}.

\vspace{-3mm}

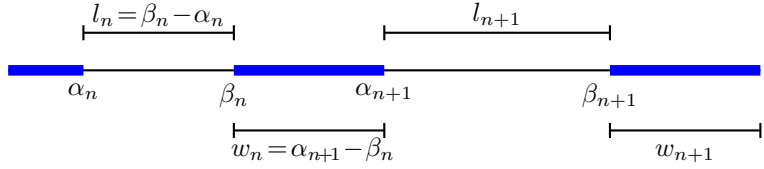
\begin{figure}[h]
\begin{center}
\psset{unit=1cm}
\begin{pspicture}(10,2)
\psline[linestyle=solid](0,1)(10,1)
\psline[linestyle=solid,linewidth=4pt,linecolor=blue](0,1)(1,1) 
\psline[linestyle=solid,linewidth=4pt,linecolor=blue](3,1)(5,1) 
\psline[linestyle=solid,linewidth=4pt,linecolor=blue](8,1)(10,1)
\uput[d](1,1){$\alpha_n$}
\uput[d](3,1){$\beta_n$}
\uput[d](5,1){$\alpha_{n+1}$}
\uput[d](8,1){$\beta_{n+1}$}
\psline[linestyle=solid](1,1.5)(3,1.5)
\psline[linestyle=solid](1,1.4)(1,1.6)
\psline[linestyle=solid](5,1.5)(8,1.5)
\psline[linestyle=solid](3,1.4)(3,1.6)	
\psline[linestyle=solid](5,1.4)(5,1.6)
\psline[linestyle=solid](8,1.4)(8,1.6)
\uput[u](2,1.4){$l_n\!=\!\be_n\!-\!\al_n$}
\uput[u](6.5,1.4){$l_{n+1}$}
\psline[linestyle=solid](3,.2)(5,.2)
\psline[linestyle=solid](8,.2)(10,.2)
\psline[linestyle=solid](3,.1)(3,.3)
\psline[linestyle=solid](5,.1)(5,.3)
\psline[linestyle=solid](8,.1)(8,.3)
\psline[linestyle=solid](10,.1)(10,.3)
\uput[u](4.05,-0.4){$w_n\!=\!\alpha_{n\!+\!1}\!-\!\beta_n$}
\uput[u](9,-0.4){$w_{n+1}$}
\end{pspicture}
\caption{\small Unperturbed spectral gaps and spectral bands of $T$. 
\label{figure spectral bands and gaps}}
\end{center}
\end{figure}

\begin{remark}
Let $l_n:=\be_n-\al_n > 0$, $w_n:=\al_{n+1}-\be_n \ge 0$ be the length of the $n$-th spectral gap and spectral band, respectively, of $\,T$, $n\in\N$. 
Since $\sum_{j=1}^{n-1}(l_j+w_j)=\alpha_n-\alpha_1$, $n\in\N$, condition \eqref{abconst.-new} in Theorem \rref{cor. infinitely many gaps nonzero rel. bound} i) can be written as
\[
 1+\liminf_{n\to\infty}\frac{l_n}{\sum_{j=1}^{n-1}(l_j+w_j)} 
 > \frac {1+\delta_A}{1-\delta_A} \ (\ge\!1),
\]
and a sufficient condition for \eqref{c4-new} in Theorem {\rm \ref{cor. infinitely many gaps nonzero rel. bound} ii)} is 
\beq
\label{eq. asymptotic gap condition with Ln, Bn}
 \kappa_{\rm s}:=\limsup_{n\to\infty}\bigg(b_n+\frac{2}{l_n}\bigg(a_n+b_n\sum_{j=1}^{n-1}(l_j+w_j)\bigg)\bigg)<1. 
\eeq  
\end{remark}

Below, for sequences $(x_n)_{n\in\N}$, $(y_n)_{n\in\N} \subset \R$, we use the notation $x_n\lesssim y_n$ if there exists a constant $C>0$ and $N\in\N$ such that $|x_n|\leq C |y_n|$, $n\geq N$; analogously, we define $x_n\gtrsim y_n$ and we write $x_n\thickapprox y_n$ if $x_n\lesssim y_n$ and $x_n\gtrsim y_n$. 

\begin{example}
\label{cor. 2 to asymptotics}
Assume that the assumptions of Theorem \ref{cor. infinitely many gaps nonzero rel. bound} hold  and that $A$ has $T$-bound $0$. If 
there exist $p_1,q_1>0$, $p_2,q_2\in\R$ such that
\begin{align}
\label{as1}
l_n&\thickapprox n^{p_1}(\log n)^{p_2},\quad w_n\lesssim n^{q_1}(\log n)^{q_2}, \\
\label{as2}
 a_n&\lesssim n^{p_1}(\log n)^{p_2},\quad b_n\lesssim\min\left\{n^{-1},n^{p_1-q_1-1}(\log n)^{p_2-q_2}\right\},
\end{align}
then there exists $\eps_0\!>\!0$ such that  $T\!+\!\eps A$, $0\!\le\! \eps\!\le\!\eps_0$, has infinitely many stable spectral free strips, more precisely,
at most finitely many $($finite$)$ spectral gaps of $\,T$ do not give rise to stable spectral free strips of $\,T+\eps A$. 
\end{example}
\begin{proof}
By assumption \eqref{as1} and some elementary estimates, we \vspace{-1.5mm} have
\begin{align}
\label{eq. sum lj and wj}
&\sum_{j=1}^{n-1}l_j\lesssim n^{p_1+1}(\log n)^{p_2}, \quad
\sum_{j=1}^{n-1}w_j\lesssim n^{q_1+1}(\log n)^{q_2}.
\end{align}
By \eqref{eq. sum lj and wj} and assumption \eqref{as2}, we see that
$\kappa_{\rm s}$ defined in \eqref{eq. asymptotic gap condition with Ln, Bn} satisfies
\beq
\label{eq. kappa proof infinite gaps}
  \kappa_{\rm s} \lesssim 2\, \limsup_{n\to\infty}\left\{\frac{a_n}{n^{p_1}(\log n)^{p_2}}+b_n\left(\frac 12+ n+n^{q_1-p_1+1}(\log n)^{q_2-p_2}\right)\right\}<\infty.
\eeq
Hence $\eps_0>0$ can be chosen so small that the corresponding bound $\eps\gamma$ for the operator $\eps A$ is $<1$ and thus
\eqref{c4-new} holds for all $0\le \eps\le\eps_0$. Now Theorem {\rm \ref{cor. infinitely many gaps nonzero rel. bound}} ii) yields the claims. 
\end{proof}


\section{Symmetric and structured perturbations}
\label{Section Structured perturbations}

If the perturbation has some additional properties, we are able to tighten the stability results for spectral gaps derived in the previous sections. Here we distinguish two cases, first, we briefly consider symmetric perturbations $A$ and, secondly, we consider non-symmetric perturbations that exhibit a certain structure with respect to the spectral gap $(\alpha_T,\beta_T)$ of $T$. In the latter case, we may even allow for perturbations with $T$-bound $\ge 1$.

In the sequel, we denote the numerical range of $A$ by
\[
  W(A):=\left\{(Ax,x): x\in \dom(A),\|x\|=1\right\}.
\]

\begin{theorem}
\label{symmetric perturbation}
Let $T$ be a self-adjoint operator in a Hilbert space $\H$ and let $(\alpha_T,\beta_T)\subset \R$ be
such that $(\alpha_T,\beta_T)$ contains $m$ eigenvalues of $\,T$, counted with multiplicity. 
Let $A$ be a T-bounded symmetric operator with $T$-bound $<1$ such that \eqref{c2} holds with $a,b\geq 0$, $b<1$, and 
\begin{itemize}
\item[{\rm i)}] if  \,\eqref{non symmetric gap remains} holds, i.e.\
$\sqrt{a^2+b^2\alpha_T^2} + \sqrt{a^2+b^2\beta_T^2} < \beta_T - \alpha_T$, 
\vspace{-1mm} set 
\[
(\alpha_{T+A},\beta_{T+A}) := \Big( \alpha_T + \sqrt{a^2+b^2\alpha_T^2}, \beta_T - \sqrt{a^2+b^2\beta_T^2} \Big); 
\]
\item[{\rm ii)}] if $A$ is bounded above and $\sup W(A)+ \sqrt{a^2+b^2\beta_T^2} < \beta_T - \alpha_T$, \vspace{-1mm} set 
\[
(\alpha_{T+A},\beta_{T+A}):= \Big(\alpha_T+\sup W(A),\beta_T - \sqrt{a^2+b^2\beta_T^2}\Big);
\]
\item[{\rm iii)}] if $A$ is bounded below and  $\sqrt{a^2+b^2\alpha_T^2}-\inf W(A)<\beta_T-\alpha_T$, 
\vspace{-1mm} set
\[
(\alpha_{T+A},\beta_{T+A}):= \Big(\alpha_T+\sqrt{a^2+b^2\alpha_T^2},\beta_T+\inf W(A)\Big);
\]
\item[{\rm iv)}] if $A$ is bounded and $\sup W(A)-\inf W(A)<\beta_T-\alpha_T$, 
set
\[
(\alpha_{T+A},\beta_{T+A}):= \big(\alpha_T+\sup W(A),\beta_T+\inf W(A)\big).
\]
\end{itemize}
Then, in each case, $(\alpha_{T+A},\beta_{T+A})$ contains at most $m$ isolated eigenvalues of $\,T+A$, counted with algebraic multiplicity. In particular, if $(\alpha_T,\beta_T)$ contains no eigenvalues of $\,T$, then $(\alpha_{T+A},\beta_{T+A})$ contains no eigenvalues of $\,T+A$. 
\end{theorem} 

\begin{proof}
Denote by $E_T$ the spectral family of the self-adjoint operator $T$,~set
\[
P_1:=E_T((-\infty,\alpha_T]),\quad P_2:=E_T([\beta_T,\infty)),\quad  P_3:=E_T((\alpha_T,\beta_T)),
\]
and $\H_i:=P_i \H$, $i=1,2,3$. Then
$\H=\H_1\oplus \H_2 \oplus \H_3$ with $\dim \H_3=m<\infty$ by the assumption on $T$.
Let $T_{ii}:=P_iT|_{\H_i}$ be the part of $T$ in $\H_i$ and set 
\[
A_{ij}:=P_iAP_j|_{\H_j},\quad \dom(A_{ij}):=P_j\dom(T),\quad i,j=1,2,3.
\]
Then the operators $T_{ii}$, $i=1,2$, are semi-bounded, $T_{11} \le \alpha$ and $T_{22} \ge \beta$.
Since $A$ is symmetric and $T$-bounded, the operators $A_{ii}$ are symmetric and $T_{ii}$-bounded, $i=1,2$; in particular, \eqref{c2} holds for the pairs $T_{ii}$, $A_{ii}$, $i=1,2$ with the same constants $a,b\geq 0$, $b<1$. Hence, by 
Corollary \ref{cor. semi-bounded}, for all $x\in P_1\dom(T)$, $y\in P_2\dom(T)$,
\beq
\label{twospaces}
\begin{split}
((T+A)x,x)&=((T_{11}+A_{11})x,x)\leq \big( \alpha_T + \sqrt{a^2+b^2\alpha_T^2} \big) 
\,\|x\|^2, 
\\ 
((T+A)y,y)&=((T_{22}+A_{22})y,y)\geq \big( \beta_T - \sqrt{a^2+b^2\beta_T^2} \big) \ \|y\|^2. 
\end{split}
\eeq
Now we are precisely in the position of \cite[Satz~8.28]{Weid} from which the claim in case~i) follows. 
In cases ii), iii), and  iv) we replace the first, the second, or both, respectively estimates in 
\eqref{twospaces} by corresponding estimates in terms of the numerical range of $A$, e.g.\ for 
$x\in P_1\dom(T)$,
\[
 ((T_{11}+A_{11})x,x)\leq \big(\alpha_T+\sup W(A)\big)\,\|x\|^2.\qedhere
\]
\end{proof}

\vspace{2mm}

In the previous theorem, the operator matrix representation of $T+A$ with respect to the ``almost spectral gap" $(\alpha_T,\beta_T)$ was used as a tool in the proof. 
In the following results, we assume that the perturbation $A$ is either ``off-diagonal" or ``diagonal" with respect to the spectral gap $(\alpha_T,\beta_T)$.

\begin{theorem}
\label{thm. diag. dom. bom.}
Let $T$ be a self-adjoint operator in a Hilbert space $\H$ with spectral family $E_T$ and spectral gap $(\alpha_T,\beta_T)$ such that $0\in (\al_T,\be_T)$, and
denote by $T_{11}$, $T_{22}$ the restrictions of $\,T$ to $\H_1:= E_T((-\infty,\alpha_T])\H$ and $\H_2:= E_T([\be_T,\infty))\H$, respectively.  
Let $A$ be $T$-bounded and off-diagonal with respect to the decomposition $\H=\H_1\oplus\H_2$, i.e.\ 
\[
  T\!=\!\begin{pmatrix}T_{11}\!&\!0\\0\!&\!T_{22}\end{pmatrix},\ \ A\!=\!\begin{pmatrix}0\!&\!A_{12}\\A_{21}\!&\!0\end{pmatrix}
  \ \ \mbox{in} \ \ %
  \dom (T) = \dom (T_{11}) \oplus \dom (T_{22}) \subset \dom (A),
\]
with $\dom(T_{11})=E_T((-\infty,\alpha_T])\dom(T)$, $\dom(T_{22})=E_T([\be_T,\infty)])\dom(T)$
in $\H=\H_1\oplus\H_2$. If the constants $a_{12}$, $a_{21}$, $b_{12}$, $b_{21}\geq 0$, in 
\begin{alignat*}{2}
 \|A_{21}x\|^2&\leq a_{21}^2\|x\|^2+b_{21}^2\|T_{11}x\|^2,&\quad &x\in \dom(T_{11}), \\ 
 \|A_{12}y\|^2&\leq a_{12}^2\|y\|^2+b_{12}^2\|T_{22}y\|^2,&\quad &y\in \dom(T_{22}), 
\end{alignat*}
\vspace{-1mm} satisfy 
\begin{equation}
\label{dec23}
 b_{12}b_{21} < 1, \quad 
 \sqrt{(a_{12}^2 + b_{12}^2\be_T^2)(a_{21}^2 + b_{21}^2\alpha_T^2)} < \Big( \frac{\be_T-\al_T}2 \Big)^2
\end{equation}
and we set
\[
\delta_{T+A}
:= \frac{\be_T-\al_T}2 - \sqrt{ \Big( \frac{\be_T-\al_T}2 \Big)^2 
- \sqrt{(a_{12}^2 + b_{12}^2\be_T^2) (a_{21}^2 + b_{21}^2\alpha_T^2)} }, 
\]
then
\beqnt
\label{eq. gap diag. dom BOM}
\sigma(T+A)\cap \big\{z\in\C : \alpha_T + \de_{T+A} < \re z < \be_T- \de_{T+A} \big\}
= \emptyset.
\eeqnt
\end{theorem}

\begin{proof}
The proof relies on the Frobenius-Schur factorization  (see e.g.\ \cite[(2.2.11)]{CT}),
\beqnt
T+A-\lm=\begin{pmatrix}I&0\\A_{21}(T_{11}-\lm)^{-1}&I\end{pmatrix}\begin{pmatrix}T_{11}-\lm&0\\0&S_2(\lm)\end{pmatrix}\begin{pmatrix}I&(T_{11}-\lm)^{-1}A_{12}\\0&I\end{pmatrix}
\eeqnt
valid on $\dom(T)= \dom (T_{11}) \oplus \dom (T_{22})$; here $S_2(\lm)$, $\lm\!\in\! \rho(T_{11})$, is the second Schur com\-plement of the operator matrix $T+A$, given by
\[
 S_2(\lm):=T_{22}-\lm-A_{21}(T_{11}-\lm)^{-1}A_{12},\quad \dom(S_2(\lm))=\dom(T_{22}).
\]
It is easy to see that (see e.g.\  \cite[\!Corollary~2.3.5]{CT}), for $\lm\!\in\!(\alpha_T\!,\be_T) \!\subset\! \rho(T_{1\!1}) \!\cap\! \rho(T_{22})$,
\beq
\label{eq. spectrum Schur complement}
\lm\in\rho(T+A)  \,\Longleftrightarrow\, 
0\in\rho(S_{2}(\lm)) \,\Longleftrightarrow\, 
1\in\rho\big(A_{21}(T_{11}-\lm)^{-1}A_{12}(T_{22}-\lm)^{-1}\big).
\eeq
Using \eqref{ineq1}, $T_{11} \le \alpha_T$, $T_{22} \geq \be_T$, and $0\in \rho(T) =\rho(T_{11}) \cap \rho(T_{22})$ by assumption, 
we obtain that, for $\lm \in \C$ with $\al_T<\re\lm<\be_T$, 
\begin{align}
\|A_{21}(T_{11}-\lm)^{-1}\| &\leq \|A_{21}(T_{11}-\re\lm)^{-1}\| \leq
\|A_{21}T_{11}^{-1}\| \,\|T_{11}(T_{11}-\re\lm)^{-1}\| \nonumber \\ 
& \leq \frac{\sqrt{a_{21}^2 \!+\! b_{21}^2 \alpha_T^2}}{|\al_T|} \!\left(1\!+\!\frac{|\re\lm|}{\re\lm \!-\! \al_T}\right)^2
= \frac{\sqrt{a_{21}^2 + b_{21}^2\alpha_T^2}}{\re\lm - \al_T}, 
\label{jan18a}\\[-3mm]
\intertext{and, \vspace{-2mm} analogously,}
\|A_{12}(T_{22}-\lm)^{-1}\| &   \leq \frac{\sqrt{a_{12}^2 + b_{12}^2\be_T^2}}{\be_T-\re\lm}.
 \label{jan18b}
\end{align}
Thus $\|A_{21}(T_{11}-\lm)^{-1}A_{12}(T_{22}-\lm)^{-1}\|<1$, and hence $\lm \in \rho(T+A)$ by \eqref{eq. spectrum Schur complement}, if
\[
 (\re\lm - \al_T) (\be_T-\re\lm) > \sqrt{ (a_{12}^2 + b_{12}^2\be_T^2) (a_{21}^2 + b_{21}^2\alpha_T^2) },
\]
which is equivalent to $\al_T+\de_{T+A} < \re\lm < \be_T - \de_{T+A}$. 
\end{proof}

\begin{remark}
If we do not assume that $0\in(\al_T,\be_T)$, the formulation of Theorem~\ref{thm. diag. dom. bom.} has to be modified.
Then four cases have to be distinguished since, instead of the estimates \eqref{jan18a}, \eqref{jan18b}, we now have to 
\vspace{-2mm}use
\begin{align*}
\|A_{21}(T_{11}\!-\!\lm)^{-1}\|&\leq 
\|A_{21}(T_{11}\!-\!\re\lm)^{-1}\|\leq
\max\left\{ b_{21}, \frac{\sqrt{a_{21}^2+b_{21}^2\alpha_T^2}}{\re \lm - \alpha_T} \right\}\!, \ \, \re\lm \!>\! \al_T,\\
\|A_{12}(T_{22}\!-\!\lm)^{-1}\|&\leq 
\|A_{12}(T_{22}\!-\!\re\lm)^{-1}\|\leq 
\max\left\{ b_{12}, \frac{\sqrt{a_{12}^2+b_{12}^2\be_T^2}}{\be_T - \re \lm} \right\}\!, \ \ \re\lm\!<\!\be_T,
\end{align*}
and the maximum may be either of the two expressions depending on the position of $\lm$.  
In all cases, \eqref{dec23} is a necessary condition. To obtain a sufficient condition, the second assumption in \eqref{dec23} has to be modified for the four different cases; we omit the tedious details. 
\end{remark}

\begin{remark}
Theorem \ref{thm. diag. dom. bom.} extends \cite[Theorem~5.4]{MR2865428} where $A$ was assumed to be bounded; in this case, $b_{12}=b_{21}=0$ so that the formula for $\delta_{T+A}$ simplifies and coincides with the corresponding formula therein due to \cite[(5.9)]{MR2865428}.
\end{remark}

\smallskip

In the next theorem the unperturbed operator $T$ is bounded below and diagonal with respect to some decomposition of the Hilbert space $\H$, while the perturbation $A$ is off-diagonal; in the second theorem, $T$ has a spectral gap symmetric to $0$ and is off-diagonal, while $A$ is diagonal.
We mention that such operators $T$ are also called \emph{abstract Dirac operators} \cite[Section~5.1]{Th}; in the first case $T$ is called \emph{even} or \emph{bosonic}, in the second case \emph{odd} or \emph{fermionic}.

\begin{theorem} 
\label{posbound}
Let $T$ be a non-negative self-adjoint operator in a Hilbert space $\H$ 
and let $A$ be $T$-bounded. Suppose that $\tau$ is a self-adjoint involution in $\H$ 
with $T\tau=\tau T$ and $A\tau=-\tau A$, i.e.
\[
T=\begin{pmatrix}T_{11}&0\\0&T_{22}\end{pmatrix}, \quad
A=\begin{pmatrix}0&A_{12}\\ A_{21}&0\end{pmatrix}  \quad \mbox{in }\ \H=\H_+\oplus\H_-,
\] 
where $H_\pm\!:=\!\Ran P_\pm$, $P_{\pm}\!:=\!\frac 12 (I\pm\tau)$. 
Let $\be_{T\!,i} \!:=\! \min \sigma(T_{ii})$, $i=1,2$, and 
$a_{12}$, $a_{21}$, $b_{12}$, $b_{21}\geq 0$ with
\begin{equation}
\label{non-symm-rel-bdd-1}
\begin{array}{rll}
 \|A_{21}x\|^2\!\!\!\!\!&\leq a_{21}^2\|x\|^2+b_{21}^2\|T_{11}x\|^2,\quad &x\in \dom(T_{11})=P_+\dom(T),\\[1mm]
 \|A_{12}y\|^2\!\!\!\!\!&\leq a_{12}^2\|y\|^2+b_{12}^2\|T_{22}y\|^2,\quad &y\in \dom(T_{22})=P_-\dom(T).
\end{array}
\end{equation}
If $\,b_{12}b_{21}<1$, \vspace{-1.5mm} then
\[
 \re\,\sigma(T+A)\geq  \min\{\be_{T\!,1},\be_{T\!,2}\}- \de_{T\!+\!A}^\oplus
\vspace{-1.5mm} 
\]
\vspace{-2.5mm}where 
\[
 \de^{\oplus}_{T\!+\!A} \!\!:=\!\! \sqrt{\!\sqrt{\!a_{1\!2}^2\!\!+\!b_{1\!2}^2 \be_{T\!,2}^2}\sqrt{\!a_{21}^2\!\!+\!b_{21}^2 \be_{T\!,1}^2}}
 \tan\!\!\left(\!\!\frac{1}{2}\!\arctan\!\frac{2\sqrt{\!\sqrt{a_{12}^2\!\!+\!b_{12}^2
 \be_{T\!,2}^2}\sqrt{a_{21}^2\!\!+\!b_{21}^2 \be_{T\!,1}^2}}}{\!\max\{\be_{T\!,1},\be_{T\!,2}\}\!-\!\min\{\be_{T\!,1},\be_{T\!,2}\}\!}\!\right)\!\!.
\]
\end{theorem}

\begin{proof}
If $0\in\rho(T)$, then the proof is similar to the proof of Theorem \ref{thm. diag. dom. bom.}. 
Now, instead of \eqref{jan18a}, \eqref{jan18b}, 
we have to use the two inequalities
\begin{align}
\hspace{-3mm} \|A_{21}(T_{11}\!-\!\lm)^{-1}\| \leq \frac{\sqrt{a_{21}^2 + b_{21}^2 \be_{T\!,1}^2}}{\be_{T\!,1} -\re\lm}, 
\quad 
\|A_{12}(T_{22}\!-\!\lm)^{-1}\| \leq \frac{\sqrt{a_{12}^2 + b_{12}^2 \be_{T\!,2}^2}}{\be_{T\!,2}-\re\lm}
\label{jan21}
\end{align}
for $\lm < \min\{\be_{T\!,1},\be_{T\!,2}\}$. Hence $\lm \in \rho(T+A)$ if $\|A_{21}(T_{11}-\lm)^{-1}A_{12}(T_{22}-\lm)^{-1}\|<1$ which holds if 
\begin{equation}
\label{jan21a}
 \lm < \frac{\be_{T\!,1}+\be_{T\!,2}}2 - \sqrt{\Big(\frac{\be_{T\!,1}-\be_{T\!,2}}2\Big)^2 + \sqrt{a_{12}^2\!\!+\!b_{12}^2 \be_{T\!,2}^2}\sqrt{a_{21}^2\!\!+\!b_{21}^2 \be_{T\!,1}^2}};
\end{equation}
elementary identities for solutions of quadratic equations (comp.\ \cite[Lemma 1.1]{KMM07}) 
show that the right hand side above is equal to $\min\{\be_{T\!,1},\be_{T\!,2}\} - \de_{T\!+A\!}^\oplus$.

If $0\notin\rho(T)$, we consider $T_{\eps}:=T+\eps \,I$, $\eps>0$, instead of $T$. By what we just proved, ${\rm Re}\,\sigma(T_{\eps}+A)\geq  \min\{\be_{T\!,1},\be_{T\!,2}\}+\eps- \de^\oplus_{T\!+\!A,\eps}$ where $\de^\oplus_{T\!+\!A,\eps}$ is obtained from $\de^\oplus_{T\!+\!A}$ by replacing $\beta_{T\!,i}$ by $\beta_{T\!,i}+\eps$ and $a_{ij}$ by $a_{ij}+\eps$. Since $\de_{T\!+\!A,\eps}^\oplus\to \de_{T\!+\!A}^\oplus$, $\eps\to 0$, the claim follows by a limiting argument.
\end{proof}

\begin{remark}
Theorem \ref{posbound} generalizes, and gives a different proof of, \cite[Lemma~1.1]{KMM07} by V.\ Kostrykin, K.A.\ Makarov, and A.K.\ Motovilov for bounded $T$ and $A$, as well as of \cite[Theorem 5.6]{CT09} for unbounded $T$ and bounded $A$.
\end{remark}

\begin{theorem} 
\label{gap thm. off-diag. dom.}
Let $T$ be a self-adjoint operator in a Hilbert space $\H$ with symmetric spectral gap $(-\beta_T,\beta_T)\subset \R$ with $\beta_T>0$, i.e.\ $\sigma(T) \cap (-\beta_T,\beta_T) = \emptyset$ and $\beta_T\in\sigma(T)$, 
and let $A$ be $T$-bounded. Suppose that $\tau$ is a self-adjoint involution in~$\H$ 
with $T\tau=-\tau T$ and $A\tau=\tau A$, i.e.
\[
T=\begin{pmatrix}0&T_{12}\\T_{12}^*&0\end{pmatrix}, \quad
A=\begin{pmatrix}A_{11}&0\\0&A_{22}\end{pmatrix}  \quad \mbox{in } \ \H=\H_+\oplus\H_-,
\] 
where $H_\pm\!:=\!\Ran P_\pm$, $P_{\pm}\!:=\!\frac 12 (I\pm\tau)$. 
Let $a_{11}$, $a_{22}$, $b_{11}$, $b_{22}\geq 0$ be such that
\begin{equation}
\label{non-symm-rel-bdd}
\begin{array}{rll}
 \|A_{11}x\|^2\!\!\!\!\!&\leq a_{11}^2\|x\|^2+b_{11}^2\|T_{12}^{*}x\|^2,\quad &x\in \dom(T_{12}^*)=P_+\dom(T),\\[1mm]
 \|A_{22}y\|^2\!\!\!\!\!&\leq a_{22}^2\|y\|^2+b_{22}^2\|T_{12}y\|^2,\quad &y\in\dom(T_{12})= P_-\dom(T).
\end{array}
\end{equation}
\vspace{-2mm}If
\begin{equation}
\label{jan04}
b_{11}b_{22}<1 \quad \mbox{and} \quad 
\sqrt{a_{11}^2+b_{11}^2\be_T^2} \sqrt{a_{22}^2+b_{22}^2\be_T^2} < \be_T^2,
\end{equation}
then $T\!+\!A$ has a stable spectral free strip $(-\beta^{\oplus}_{T\!+\!A},\beta^{\oplus}_{T\!+\!A})\!+\!\I\R \!\subset\! \C$, 
more precisely,
\beq
\label{noel15}
\sigma(T+sA)\cap\SSet{z\in\C}
{|\re z|<\beta^{\oplus}_{T+A}} = \emptyset, \quad s\in[0,1],
\eeq
with $ \beta^{\oplus}_{T+A}\in (0,\beta_T]$ given by
\[
  \beta^{\oplus}_{T+A}\!:=\!\sqrt{\be_T^2\!+\!\left(\frac{\sqrt{a_{11}^2\!+\!b_{11}^2\be_T^2}\!-\!\sqrt{a_{22}^2\!+\!b_{22}^2\be_T^2}}{2}\right)^{\!\!2}}
  \!-\frac{\sqrt{a_{11}^2\!+\!b_{11}^2\be_T^2}\!+\!\sqrt{a_{22}^2+b_{22}^2\be_T^2}}{2}.
\]
\end{theorem}

\begin{proof}
The assumptions on $T$ imply that $0\in \rho(T_{12}^*)$ and $\|T_{12}^{-*}\|=\frac 1{\beta_T}$. Hence we can use the factorization (see \cite[(2.2.16)]{CT})
\[
T+A-\lm=\begin{pmatrix}
I&(A_{11}-\lm)T_{12}^{-*}\\0&I
\end{pmatrix}
\begin{pmatrix}
0&Q_2(\lm)\\T_{12}^*&0
\end{pmatrix}
\begin{pmatrix}
I&T_{12}^{-*}(A_{22}-\lm)\\0&I
\end{pmatrix}
\]
valid on $\dom(T)$; here $Q_{2}(\lm)$ is the so-called second quadratic complement of the operator matrix $T+A$ (see \cite[Definition 2.2.20]{CT09}) given by
\[
Q_2(\lm)\!:=\!T_{12}-(A_{11}\!-\!\lm)T_{12}^{-*}(A_{22}\!-\!\lm),\quad \dom(Q_2(\lm)):=\dom(T_{12}), 
\]
for $\lm\in\C$. By \cite[Corollary 2.3.8]{CT09}, we have
\begin{equation}
\label{christmas15}
\lm\in\rho(T\!+\!A)\iff 0\in\rho(Q_2(\lm))\iff 1\in\rho\left((A_{11}\!-\!\lm)T_{12}^{-*}(A_{22}\!-\!\lm)T_{12}^{-1}\right)\!.
\end{equation}
By \eqref{non-symm-rel-bdd}, we can estimate 
\[\begin{split}
 \|(A_{11}-\lm)T_{12}^{-*}(A_{22}-\lm)T_{12}^{-1}\|
 &\leq (\|A_{11}T_{12}^{-*}\|+|\lm|\|T_{12}^{-*}\|)(\|A_{22}T_{12}^{-1}\|+|\lm|\|T_{12}^{-1}\|)\\
  &\leq\left(\frac{\sqrt{a_{11}^2\!+\!b_{11}^2\be_T^2}}{\be_T}+\frac{|\lm|}{\be_T}\right)
 \left(\frac{\sqrt{a_{22}^2\!+\!b_{22}^2\be_T^2}}{\be_T}+\frac{|\lm|}{\be_T}\right)\!.
\end{split}\]
If $|\lm|< \beta_{T+A}^{\oplus}$, then this upper bound is $<1$ and hence $\lm\in\rho(T+A)$ by \eqref{christmas15}. 
Note that $\beta_{T+A}^{\oplus} \le \beta_T$ and that $\beta_{T+A}^{\oplus}>0$ due  to assumption \eqref{jan04}.
\vspace{-3mm}
\end{proof}

\vspace{2mm}

\begin{remark}
Theorem \ref{posbound} refines and extends Corollary \ref{cor. semi-bounded}, and hence Kato's semiboundedness stability result \cite[Theorem V.4.11]{Ka},
while Theorem \ref{gap thm. off-diag. dom.} refines and extends Corollary \ref{symmetric gap}, i.e.\ Theorem \ref{new-non-symmetric gap} for a symmetric spectral gap.
In fact, since different relative boundedness constants may be chosen for $A_{11}$, $A_{22}$ in \eqref{non-symm-rel-bdd-1} and \eqref{non-symm-rel-bdd}, respectively,
we only have to impose conditions on the products in \eqref{jan04} and not on the factors; e.g.\ we can allow for perturbations with relative bound $>1$ if the product of the column-wise relative bounds is $<1$. If we choose $a_{11}\!=\!a_{22}\!=:\!a$, $b_{11}\!=\!b_{22}\!=:\!b$, and in Theorem \ref{posbound} $\beta_{T\!,1}=\beta_{T\!,2}=:\beta_T$, then $\delta _{T\!+A\!}^\oplus = \sqrt{a^2+b^2\beta_T^2}$ and $\beta_{T\!+A\!}^\oplus = \beta_{T+A} = \beta_T - \sqrt{a^2+b^2\beta_T^2}$;
so in this case Theorem~\ref{posbound} and Theorem \ref{gap thm. off-diag. dom.} coincide with Corollary \ref{cor. semi-bounded} 
and Corollary~\ref{symmetric gap},~respectively.
\end{remark}



\section{Applications}
\label{Section Applications}

In this final section we illustrate our results by several applications, including massless and massive Dirac operators with complex potentials, point-coupled periodic systems on manifolds with infinitely many spectral gaps, and two channel scattering Hamiltonians with dissipation.

\subsection{Massless and massive Dirac operators}

The massive Dirac operator is the prototype of a non-semi-bounded self-adjoint operator with spectral gap. Recently, the massless Dirac operator in two dimensions has gained particular interest since it appears as an effective Hamiltonian near the so-called Dirac cones in single layer graphene, see \cite{GeimNovoselov}. 

Using natural units (i.e.\ $\hbar=c=1$), the Dirac operator $H_0$ of a free relativistic particle of mass $m\ge 0$  in dimension $d=2$ or $d=3$ is 
the self-adjoint operator given by
\begin{align}\label{Dirac operator}
H_0=
\begin{cases}
-\I \sum_{j=1}^{2}\sigma_j\,\partial_j,\quad &  \dom(H_0)=W^{2,1}(\R^2,\C^2), \qquad d=2,\\[2mm]
-\I \sum_{j=1}^{3}\al_j\,\partial_j+m\,\be, \quad & \dom(H_0)=W^{2,1}(\R^3,\C^4), \qquad d=3,
\end{cases}
\end{align}  
in $L^2(\R^2,\C^2)$ and $L^2(\R^3,\C^4)$, respectively. Here 
$\sigma_1$, $\sigma_2$, $\sigma_3 \in M(2,\C)$ and $\al_1$, $\al_2$, $\al_3$, $\be \in M(2,\C)$ 
are the Pauli and Dirac matrices, respectively.
It is well-known that $H_0$ has purely absolutely continuous spectrum $\sigma(H_0)=(-\infty,-m]\cup[m,\infty)$, see e.g.\ \cite{Th}; in particular, $H_0$ has a spectral gap if and only if $m > 0$.

\medskip

\noindent
\textbf{Massless Dirac operator in $\R^2$.} 
Assume that $V:\R^2\to M(2,\C)$ is a measurable matrix-valued function with
\[
V\in L^p(\R^2,M(2,\C))\quad \mbox{for some $p>2$}.
\]
Using H\"older and Hausdorff-Young inequality (see e.g.\ the proof of \cite[Satz 17.7]{Weid2}), one obtains the family of inequalities
\beq\label{eq. Weidmann inequality Dirac}
\|Vf\|^2\leq a_{p}(t)^2\|f\|^2+b_{p}(t)^2\|H_0 f\|^2,\quad t>0,
\eeq
of the form \eqref{c2} with
\begin{align*}
a_{p}(t):=C_p t^{-\frac{2}{p}},\quad b_{p}(t):=C_p t^{\frac{p-2}{p}},\quad C_p:=\|V\|_p\,(2\pi)^{-2/p}\left(\frac{2\pi}{p-2}\right)^{1/p}.
\end{align*}
Since $b_p$ is a strictly monotonically increasing function, we can solve for $t$ and re-parametrise~$a_p$; the family of inequalities \eqref{eq. Weidmann inequality Dirac} can then be written as
\beq
\label{last!}
\|Vf\|^2\leq a_{p}(b_p )^2\|f\|^2+b_p ^2\|H_0 f\|^2,\quad a_p(b_p ):=C_p^{\frac{p}{p-2}}b_p ^{-\frac{2}{p-2}},\quad b_p >0.
\eeq
Theorem \ref{new-non-symmetric gap} i) yields that 
\[
\sigma(H_0+V)\subset\bigcap_{0<b_p <1}\SSet{z\in\C}{|\im z|^2\leq \frac{a_{p}(b_p )^2+b_p ^2|\re z|^2}{1-b_p ^2}}.
\]
The envelope of the set on the right hand side may be computed by solving the system of equations
\begin{alignat*}{2}
(1-b_p ^2)y&=a_{p}(b_p )^2+b_p ^2x, \quad 
-b_p y&=a_{p}(b_p )\,a_{p}'(b_p )+b_p x
\end{alignat*}
for the variables $x\!=\!|\re z|^2\!$, $y\!=\!|\im z|^2\!$. The re\-sulting curve has the parametrization
\begin{align*}
x(b)&=\left(\frac{C_p}{b_p }\right)^{\frac{2p}{p-2}}\frac{2-b_p ^2 p}{p-2}, \quad  
y(b)=\frac{p\,C_p^2}{p-2}\left(\frac{C_p}{b_p }\right)^{\frac{4}{p-2}}, \qquad 0<b_p<1.
\end{align*}
The region bounded by this curve, which contains the spectrum of the Dirac operator $H_0+V$ with $V\in L^p(\R^2,M(2,\C))$, is displayed in
Figure \ref{fig. massless} for $p=5$ and $p=7$. 
Note that larger $p$ give tighter estimates for the spectrum; asymptotically, i.e.\ for $b_p\searrow 0$, the curves are of the form 
$$
|\im z|=C_p\left(\frac{p}{p-2}\right)^{\frac{p-2}{2p}} \left(\frac p2 \right)^{\frac 1p} |\re z|^{\frac{2}{p}}
= \sqrt{\frac{p}{p-2}} (4\pi)^{-\frac 1p} \|V\|_p \,|\re z|^{\frac{2}{p}}.
$$
In the limit case $p \!\to\! \infty$, the region bounded by the above curves degenerates into the horizontal strip 
$|\im z| \!\le\! \|V\|_\infty$; for $p \!\searrow\! 2$ it degenerates into the complex \vspace{-2mm}plane.

\begin{figure}[h]\label{fig. massless}
\begin{center}
\includegraphics[scale=.17]{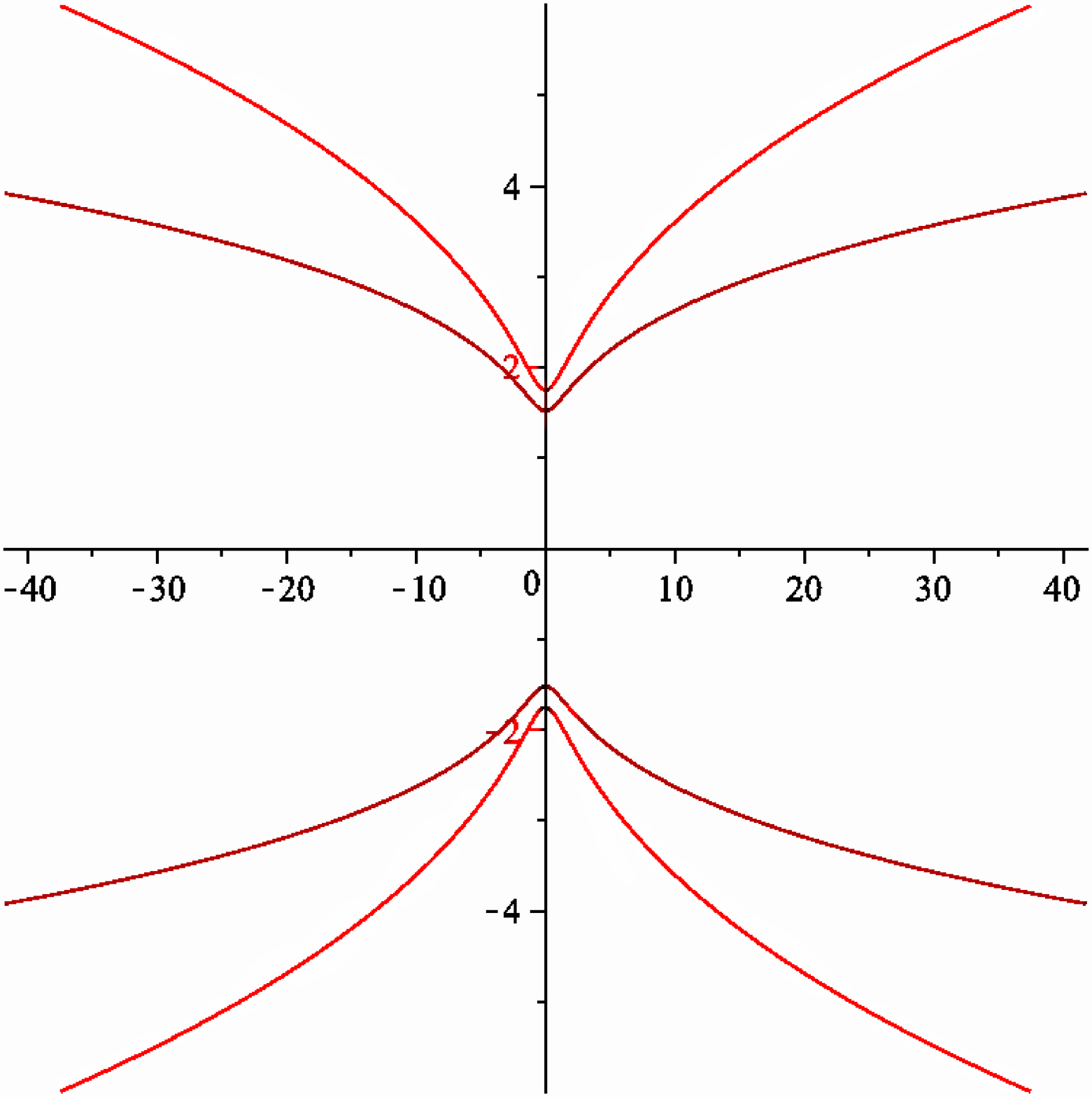} 
\vspace{-1mm}
\end{center}
\caption{Curves bounding the spectrum of the massless Dirac operator in $\R^2$ for $p=5$ (outer curve)  and $p=7$ (inner curve).}
\end{figure}

\medskip

\noindent
\textbf{Massive Dirac operator with Coulomb-like potential in  {\boldmath $3$} dimensions.}
Assume that $V:\R^3\to M(4,\C)$ is a measurable (possibly non-hermitian) matrix-valued function such that
\[
\|V(x)\|_{M(4,\C)}^2\leq C_1^2+C_2^2\,|x|^{-2}\quad \mbox{for almost all } x\in\R^3,
\]
where $\|\cdot\|_{M(4,\C)}$ denotes the operator norm of the matrix $V(x)$ in $\C^4$ equipped with the Euclidean scalar product, and $C_1$, $C_2\geq 0$.
Theorem \ref{symmetric gap} and Hardy's inequality imply that, if $\sqrt{C_1^2+4 C_2^2m^2}<m$, 
then $H=H_0+V$ is bisectorial and 
\begin{align*}
 \sigma(H)\subset&\SSet{z\in\C}{|\re z|\geq m-\sqrt{C_1^2+4 C_2^2 m^2},\,|\im z|^2\leq \frac{C_1^2+4 C_2 |\re z|^{2}}{1-4 C_2^2}}.
\end{align*}
Thus the spectrum of $H=H_0+V$ remains separated into two parts; in particular, if $V$ is hermitian and hence $\sigma(H)\subset\R$, the electronic and positronic part of the spectrum of the Dirac operator remain separated. 

\begin{remark}
i) Analogous spectral enclosures hold for Dirac operators in $\R^d$. 

ii) The problem of decoupling the corresponding electronic and positronic spectral subspaces, i.e.\ 
of finding a generalized Foldy-Wouthuysen transformation for Dirac operators in $\R^d$ with unbounded potentials, was studied in~\cite{Cuenin11}.

iii) For $d=1$, the Birman-Schwinger principle together with the explicit form of the resolvent kernel of~$H_0$ were used to prove more refined estimates for the \emph{eigenvalues} of Dirac operators with non-hermitian potentials in~\cite{MR3177918}.
\end{remark}

\subsection{Point-coupled periodic systems on manifolds}
\label{subsec:7.2} 

The spectrum of Schr\"o\-din\-ger operators in $\R^d$ with periodic potentials typically exhibits a band-gap structure. 
Generically, the number of gaps is infinite for $d=1$ and finite for $d>1$; this so-called Bethe-Sommerfeld conjecture was proved for a general class of smooth
potentials for $d=2,3$, see e.g.~\cite{MR518334,MR549367}. 

However, in the case of singular potentials, the number of gaps may be infinite also for $d\!>\!1$. 
Moreover, while the band to gap ratio generally tends to $\infty$ at high energies, 
the situation is reversed for free quantum motion on periodic manifolds, as shown in \cite{BEG03}. 
The manifolds consist of two-dimensional spheres,
connected either by points where two spheres touch (case~1) or by line segments (case~2). 

The Hilbert space of the system is an infinite number of copies of $L^2(X,\C)$ where $X$ is the two-sphere $\mathbb{S}^2$ (case 1) or $\mathbb{S}^2$ with a line segment $[0,d)$ attached (case 2). On this Hilbert space we consider the operator
\begin{align*}
\widehat{S}=\bigoplus_{m\in\Z}(-\Delta_{X}) \quad \mbox{in} \ \ \H := \bigoplus_{m\in\Z} L^2(X,\C), 
\\[-9mm]
\end{align*}
\vspace{-1mm}where
\begin{alignat*}{3}
&\Delta_{X}:=-\Delta_{\mathbb{S}^2} &&\quad \mbox{in} \ \ L^2(\mathbb{S}^2,\C) &&\quad\mbox{(case 1),}\\[-1mm]
&\Delta_{X}:=(-\Delta_{\mathbb{S}^2})\oplus \left(-\frac{\rd^2}{\rd x^2}\right) &&\quad \mbox{in} \ \ L^2(\mathbb{S}^2,\C)\oplus L^2([0,d),\C) &&\quad\mbox{(case 2),}
\end{alignat*}
and $-\Delta_{\mathbb{S}^2}$ is the Laplace-Beltrami operator on $\mathbb{S}^2$. 
With domain of $\Delta_{X}$ chosen as the set of functions in $W^{2,2}(X,\C)$  that vanish at two distinct points on $\mathbb{S}^2$ (case~1) or at one point of $\mathbb{S}^2$ and at $d$ (case~2), $\widehat{S}$ is a symmetric operator with infinite deficiency indices. 

If a self-adjoint extension $S$ of $\widehat{S}$ is selected by means of local boundary conditions reflecting the necklace geometry of the system, see \cite[Section~3]{BEG03} for details,
then the gap and band lengths $l_n$ and $w_n$ of the spectrum of $S$ satisfy 
\begin{alignat}{4}\label{eq. band gap periodic manifolds}
&l_n \thickapprox n^2,\quad &&w_n\thickapprox n^2(\log n)^{-\epsilon} &&\quad\mbox{(case 1)},\\
\label{eq. band gap periodic manifolds-2}
&l_n \thickapprox n^2,\quad &&w_n\thickapprox n^{2-\epsilon} &&\quad\mbox{(case 2),}
\end{alignat}
for some $\epsilon\in(0,1)$. 

Consider a perturbation $V=\bigoplus_{m\in\Z}V_m$ with potentials $V_m:X\to\C$ such that, in both cases, ${\rm supp}\,V_m \subset \mathbb{S}^2$ and
\begin{equation}
\label{pot}
c:=\sup_{m\in\Z}\|V_m\|_{L^p(\mathbb{S}^2)}<\infty \quad \mbox{ for some } \ p\in(2,\infty).
\vspace{-2mm}
\end{equation}

We can estimate the relative boundedness constants of $V_m$ with respect to $-\Delta_{\mathbb{S}^2}$ in \eqref{c2} using the following interpolation inequality. The latter was established in~\cite{MR1134481} in the more general context of compact Riemannian manifolds with uniformly positive Ricci curvature; we state the inequality as presented e.g.\ in \cite{MR3011461}. If $\mu$ denotes the normalized surface measure on $\mathbb{S}^2$, then, for all $q\in(2,\infty)$,
\begin{align}\label{eq. interpolation inequality on the sphere}
\left(\int_{\mathbb{S}^2}|u|^q\,\rd\mu\right)^{2/q}\leq \frac{q-2}{2} \int_{\mathbb{S}^2}|\nabla_{\mathbb{S}^2} u|^2\,\rd\mu+\int_{\mathbb{S}^2}|u|^2\,\rd\mu,\quad u\in H^1(\mathbb{S}^2,\rd\mu).
\end{align}
It easily follows from \eqref{eq. interpolation inequality on the sphere} that, for $m\in\Z$, $n\in\N$, and $u\in\dom(-\Delta_{\mathbb{S}^2})$,
\begin{align*}
\|V_m u\|_{L^2(\mathbb{S}^2)}^2&\leq \|V_m\|_{L^p(\mathbb{S}^2)}^2\|u\|_{L^{\frac{2p}{p-2}}(\mathbb{S}^2)}^2\\
&\leq c^2(4\pi)^{-2/p}\left(\frac{1}{p-2}\|\nabla_{\mathbb{S}^2} u\|_{L^{2}(\mathbb{S}^2)}^2+\|u\|_{L^{2}(\mathbb{S}^2)}^2\right)\\
&\leq c^2(4\pi)^{-2/p}\left(\frac{1}{(p-2)n^2}\|\Delta_{\mathbb{S}^2} u\|_{L^{2}(\mathbb{S}^2)}^2+\left(1+\frac{n^2}{p-2}\right)\|u\|_{L^{2}(\mathbb{S}^2)}^2\right).
\end{align*}
Hence, the relative boundedness constants of $V$ with respect to $S$ in \eqref{c2} satisfy 
\[
  a_n\thickapprox n^2,  \quad b_n\thickapprox n^{-2}.
\]
Thus we are in the situation of Example \ref{cor. 2 to asymptotics} which shows that, if the potential satisfies~\eqref{pot}, then
there exists $\eps_0\!>\!0$ such that  $S\!+\!\eps V$, $0\!\le\! \eps\!\le\!\eps_0$, has infinitely many stable spectral free strips; more precisely,
at most finitely many $($finite$)$ spectral gaps of $\,S$ do not give rise to stable spectral free strips of $\,S+\eps V$.

\subsection{Two-channel scattering with dissipation}

The Hamiltonian of a non-relativistic two-channel potential scattering model in $\R^d$ is given by
(in the centre of mass frame)
\beqnt
H=\begin{pmatrix}
   -\Delta+x^2&V_{12}\\ V_{21}&-\Delta 
  \end{pmatrix}
=:\begin{pmatrix}
   H_C&V_{12}\\ V_{21}&H_S
  \end{pmatrix}
\eeqnt 
in the Hilbert space $L^2(\R^d,\C)\oplus L^2(\R^d,\C)$
where all masses, as well as the Planck constant $\hbar$, have been set to unity. We assume that the Hamiltonian $H_C=-\Delta+x^2$ in the confined channel, governing the relative motion between two permanently confined particles (e.g.\ a quark and an antiquark), is a harmonic oscillator, while $H_S = -\Delta$ in the scattering channel is the free Laplacian.

Then both $H_C$ and $H_S$ are self-adjoint in $L^2(\R^d)$ and bounded below. The spectrum of $H_C$ is discrete, $\sigma(H_C)= \SSet{2n+d}{n\in\N_0}$,
while the spectrum of $H_S$ is absolutely continuous, $\sigma(H_S)=[0,\infty)$, and hence 
\begin{equation}
\label{newyear2016}
 \be_{H,1}= \min \sigma(H_C) = d, \quad \be_{H,2}= \min \sigma(H_S) = 0.  
\end{equation}

The communication between the two channels is mediated by the off-diagonal potentials $V_{12}$, $V_{21}$; if $V_{21}\neq\overline{V_{12}}$, which corresponds to a scattering process with dissipation, then the two-channel Hamiltonian $H$ is not self-adjoint.
We assume~that
\begin{align*}
V_{12}\in L^p(\R^d)\quad \mbox{for some } p>d/2,\, p\geq 2,
\end{align*}
and that there are $p_0$, $p_{1,\alpha}$, $p_{2,\alpha}\!\geq\! 0$, such that, with the usual multi-index~notation,
\begin{align}\label{eq. assumption on potential for harmonic oscillator}
|V_{21}(x)|\leq p_0+\sum_{|\alpha|=1}p_{1,\alpha}x^{\alpha}+\sum_{|\alpha|=2}p_{2,\alpha}x^{\alpha}.
\end{align}
In terms of the creation and annihilation operators
\[
A_j=\Big(\frac{\partial}{\partial x_j}+x_j\Big),\quad A_j^*=\Big(-\frac{\partial}{\partial x_j}+x_j\Big),\quad j=1,\ldots,d,
\]
the harmonic oscillator $H_C$ and the monomials $x_j$ can be written as
\begin{align}\label{eq. harmonic oscillator and x with creation/annihilation operators}
H_C=\sum_{j=1}^d A_j^*A_j+d,\quad x_j=\frac{1}{2}\left(A_{j}+A^*_{j}\right). 
\end{align}
Then, for $f\in\dom(H_C)$, 
\begin{align*}
\|H_C f\|^2=\sum_{j,k=1}^d\|A_jA_k f\|^2+2d\sum_{j=1}^d\|A_j f\|^2+d^2\|f\|^2,
\end{align*}
and hence, if $A_j^{\#}$ is either $A_{j}$ or $A^*_{j}$, 
\begin{align*}
\|A_j^{\#}A_k^{\#}f\|^2\leq \|H_C f\|^2,\quad \|A_j^{\#} f\|^2\leq \frac{1}{2d}\|H_C f\|^2, \quad \|f\|^2\leq \frac{1}{d^2}\|H_C f\|^2.
\end{align*}
This, together with \eqref{eq. assumption on potential for harmonic oscillator} and 
\eqref{eq. harmonic oscillator and x with creation/annihilation operators} implies that
\begin{align}\label{eq. b21 channel}
\|V_{21}f\|^2\leq b_{21}^2\|H_C f\|^2,\quad b_{21}^2:=\frac{2p_0^{2}}{d^2}+\frac{1}{d}\sum_{|\alpha|=1}p_{1,\alpha}^{2}+2\sum_{|\alpha|=2}p_{2,\alpha}^{2}.
\end{align}
On the other hand, in analogy to \eqref{eq. Weidmann inequality Dirac}, \eqref{last!}, for $f\in\dom(H_S)$,
\begin{align*}
\|V_{12}f\|^2 
\!\leq\! a_{12,p}^{2}(b_{12,p})\|f\|^2\!\!+\!b_{12,p}^2\|H_S f\|^2\!, 
\ \  
a_{12,p}(b_{12,p})\!:=\!C_p^{\frac{2p}{2p-d}}b_{12,p}^{-\frac{d}{2p-d}}\!,\ \ b_{12,p}\!>\!0,
\end{align*}
\vspace{-1mm} with 
\begin{align}\label{eq. cp channel}
C_p:=\|V_{12}\|_p(2\pi)^{-d/p}\frac{2\pi^{d/2}}{\Gamma(d/2)}B(d/2,p-d/2),
\end{align}
where $B$, $\Gamma$ denote the Beta- and Gamma-function.
Together with \eqref{newyear2016}, 
Theorem~\ref{posbound} yields that
\begin{align*}
{\rm Re}\,\sigma(H)\geq \,-\!\!\inf_{0< b_{12,p} < b_{21}^{-1}}\left\{ b_{21}d\,a_{12,p}(b_{12,p})\tan\!\left(\frac{1}{2}\arctan\frac{2b_{21}d\, a_{12,p}(b_{12,p})}{d}\right)\right\}.
\end{align*}
Since $a_{12,p}$ is monotonically decreasing, the infimum is approached in the limit $b_{12,p}\to b_{21}^{-1}$. 
Altogether, the two-channel Hamiltonian $H$ is m-accretive with 
\begin{align*}
{\rm Re}\,\sigma(H)&\geq -b_{21}^{\frac{2p}{2p-d}}d\,C_p^{\frac{2p}{2p-d}}\tan\!\bigg(\frac{1}{2}\arctan\frac{2b_{21}^{\frac{2p}{2p-d}}d\,C_p^{\frac{2p}{2p-d}}}{d}\bigg),
\end{align*}    
where $b_{21}$ and $C_p$ are as in \eqref{eq. b21 channel} and \eqref{eq. cp channel}, respectively. 
We mention that, even in the case of symmetric potential $V_{21}=\overline{V_{12}}$, this bound is tighter than the bound obtained from 
Kato's semiboundedness stability result \cite[Theorem V.4.11]{Ka}. 

\bigskip
{\small

\noindent
{\bf Acknowledgements.} 
{\small The first author gratefully acknowledges the support of Schweizerischer Nationalfonds, SNF, through the postdoc stipends
PBBEP2$\_136596$ and P300P2\_ 147746; the second author thanks for the support of SNF, within grant no.\ 200020$\_146477$.
}

\bibliographystyle{plain}
\bibliography{Bibliography}

}

\end{document}